\documentclass[11pt]{article}
\usepackage{amscd, amsmath, amssymb, amsthm}
\usepackage[all,cmtip]{xy}
\usepackage[pagebackref]{hyperref}

\title{Chow groups with twisted coefficients  }
\author{Burt Totaro}
\date{  }

\def\Z{\text{\bf Z}}
\def\Q{\text{\bf Q}}

\def\C{\text{\bf C}}
\def\P{\text{\bf P}}
\def\F{\text{\bf F}}

\def\inj{\hookrightarrow}
\def\imp{\Rightarrow}

\def\etale{\'etale }
\DeclareMathOperator{\et}{et}
\DeclareMathOperator{\Hom}{Hom}
\DeclareMathOperator{\Spec}{Spec}
\DeclareMathOperator{\ev}{ev}
\DeclareMathOperator{\tr}{tr}
\DeclareMathOperator{\res}{res}
\DeclareMathOperator{\im}{im}
\DeclareMathOperator{\Ind}{Ind}
\DeclareMathOperator{\codim}{codim}
\DeclareMathOperator{\Sm}{Sm}
\DeclareMathOperator{\T}{T}
\DeclareMathOperator{\M}{M}
\DeclareMathOperator{\SH}{SH}
\DeclareMathOperator{\HH}{H}
\DeclareMathOperator{\h}{\mathcal{H}}
\DeclareMathOperator{\map}{map}
\DeclareMathOperator{\fr}{fr}
\DeclareMathOperator{\Gal}{Gal}
\DeclareMathOperator{\nr}{nr}
\DeclareMathOperator{\Tor}{Tor}
\DeclareMathOperator{\Zar}{Zar}
\DeclareMathOperator{\sing}{sing}
\DeclareMathOperator{\Sh}{Sh}
\DeclareMathOperator{\coker}{coker}

\begin{document}
\maketitle
\newtheorem{theorem}{Theorem}[section]
\newtheorem{corollary}[theorem]{Corollary}
\newtheorem{lemma}[theorem]{Lemma}

\theoremstyle{definition}
\newtheorem{definition}[theorem]{Definition}
\newtheorem{example}[theorem]{Example}
\newtheorem{question}[theorem]{Question}
\newtheorem{conjecture}[theorem]{Conjecture}

\theoremstyle{remark}
\newtheorem{remark}[theorem]{Remark}

It is natural to ask whether we can define
the Chow group of algebraic cycles
with twisted coefficients, as we can do
for cohomology. Rost
gave such a definition. Namely, he defined Chow groups
with coefficients in a locally constant \etale sheaf $E$, assuming
that $E$ is torsion of exponent invertible in the base field $k$
\cite[Remarks 1.11 and 2.5]{Rost}. We generalize the definition
so that $E$ need not be torsion (and, in characteristic $p>0$,
$p$ need not act invertibly on $E$).
It was not obvious how to make such a definition,
because Chow groups
are defined in terms of the Zariski topology,
not the \etale topology. For the constant sheaf $\Z_X$,
the twisted Chow group $CH_i(X,\Z_X)$ is the usual Chow group $CH_iX$.

Chow groups with twisted coefficients
have hardly been studied at all,
although Rost's more general notion of cycle modules has been influential.
The main goal of this paper is to present some methods
for computing twisted Chow groups, in the hope of making these
groups more useful in practice. This continues the grand theme
of importing ideas from homotopy theory into algebraic geometry,
in order to understand torsion phenomena.

In particular, twisted Chow groups
are directly related to Serre's notion
of ``negligible cohomology'' for finite groups
\cite[section 26]{GMS}. This paper
was prompted by a remarkable computation of negligible cohomology
by Merkurjev and Scavia \cite{MS}, which we generalize in terms
of twisted Chow groups (Theorem \ref{transfer}).
In short, twisted Chow groups are always generated by the Chow
groups of suitable covering spaces. (For now, we are not giving
new calculations of the usual Chow groups.) Twisted Chow groups
tensored with the rationals can easily be computed in terms of
the Chow groups of suitable covering spaces (Lemma \ref{rational}),
and so their main novelty is integral
or modulo a prime number.

Guillot, Di Lorenzo, and Pirisi defined equivariant Chow groups
with coefficients in any cycle module \cite{Guillot, DLP}.
In particular, that gives
a definition of equivariant Chow groups with coefficients,
$CH^*_G(X,M)$ for a $\Z G$-module
$M$ (section \ref{equivariant}).
An interesting special case is when $X$ is a point; then we get
a definition of twisted Chow groups of the classifying space
of $G$, $CH^*(BG_k,M)$. (These groups
map to the cohomology of $G$ with coefficients
in $M$.) We prove some general bounds for generators of $CH^*(BG_k,M)$
(Theorems \ref{euler} and \ref{degreebound}).
We give complete calculations of these invariants (for all $G$-modules)
when $G$ is cyclic, quaternion, or the group $\Z/2\times\Z/2$
(Theorems \ref{cyclic}, \ref{quaternion}, \ref{klein}). The results
are related to group cohomology, but with
some striking differences (Remarks \ref{nobound}
and \ref{kleinremark}).

Heller, Voineagu, and \O stv\ae r have defined
twisted motivic cohomology \cite{HVO}. We reformulate
the definition and provide some computational tools,
such as relating twisted motivic
cohomology with twisted higher Chow groups
(section \ref{motivic}). Surprisingly,
there is a surjection from $H^{2i}_{\M}(X,E(i))$ to the twisted
Chow group $CH^i(X,E)$, but it is not always an isomorphism
(Theorems \ref{surjective} and \ref{counterexample}).
In the example, the monodromy group of $E$ is $\Z/2\times \Z/2$.
The example depends on relating twisted Chow groups
with the theory of algebraic tori, such as the notion
of coflasque resolutions.

I believe that twisted Chow groups
and twisted motivic cohomology
are both worth investigating.
Each theory has its own advantages (Remark \ref{comparison}).
More generally, we conjecture a definition of Chow groups twisted
by any birational sheaf with transfers, in the sense
of Kahn--Sujatha (Conjecture \ref{cyclemodule}).

This work was supported by NSF grant DMS-2054553 and
Simons Foundation grant SFI-MPS-SFM-00005512.
Thanks to Alexander Merkurjev and Federico Scavia
for useful conversations.

\section{Definition}
\label{definition}

We start by recalling Rost's definition of Chow groups
with twisted coefficients in the torsion case. The definition mixes
the Zariski and \etale topologies in a nontrivial way.

\begin{definition}
\label{twisted}
Let $X$ be a separated scheme of finite type
over a field $k$. Let $E$ be a locally constant \etale sheaf
on $X$ which is killed by a positive integer $r$ that is
invertible in $k$. Then $CH_i(X,E)$ is defined to be the cokernel
of the residue map on Galois cohomology groups:
$$\oplus_{x\in X_{(i+1)}} H^1(k(x),E(1))
\to \oplus_{x\in X_{(i)}} H^0(k(x),E).$$
Here $E(1)$ means the Galois module $E\otimes_{\Z/r}\mu_r$, where
$\mu_r$ denotes the $r$th roots of unity in a separable closure of $k$.
Also, $X_{(i)}$
means the set of points of the scheme $X$ whose closure has dimension $i$.
\end{definition}

For $E=(\Z/r)_X$, we have $H^1(k(x),E(1))\cong H^1(k(x),\mu_r)\cong
k(x)^*/(k(x)^*)^r$, and so $CH_i(X,\Z/r)$ is the usual Chow group
modulo $r$, $CH_i(X)/r$. {\it Throughout the paper, cohomology of fields
with no topology specified
will denote Galois cohomology (or equivalently, \etale cohomology).}
Schemes of finite type over a field will be assumed to be separated.

Even though the definition involves \etale cohomology of fields,
the twisted Chow groups of $X$
cannot be viewed as \etale cohomology of $X$,
regardless of the choice of coefficients.
In particular, for $X$ smooth over $k$, there is a natural homomorphism
from $CH^i(X,E)$ to \etale motivic cohomology $H^{2i}_{\et}(X,E(i))$
(Theorem \ref{etalecycle}), and this is an isomorphism rationally,
but not integrally. For example,
when $k$ is the complex numbers $\C$,
\etale motivic cohomology $H^{2i}_{\et}(X,\Z/r(i))$ can
be identified with ordinary cohomology, $H^{2i}(X(\C),\Z/r)$. That is
usually quite different from $CH^i(X,\Z/r)=CH^i(X)/r$.

In this paper, we generalize the definition of twisted Chow groups to any
locally constant \etale sheaf $E$ of abelian groups, not necessarily torsion,
as follows.

\begin{definition}
Let $X$ be a separated scheme of finite type
over a field $k$. Let $E$ be a locally constant \etale sheaf
on $X$. Then $CH_i(X,E)$ is defined to be the cokernel
of the residue map on Galois cohomology groups:
$$\oplus_{x\in X_{(i+1)}} H^1(k(x),E(1))
\to \oplus_{x\in X_{(i)}} H^0(k(x),E).$$
Here we interpret $\Z(1)$
as $G_m[-1]$, a shift of the multiplicative group
in the derived category of \etale sheaves over a field,
as in Voevodsky's theory of motivic cohomology
\cite[Theorem 4.1]{MVW}. Define
$E(1)$ as the derived tensor product
$E\otimes_{\Z}^L G_m[-1]$.
\end{definition}

This agrees with Rost's definition (Definition \ref{twisted})
when $E$ is torsion
of exponent invertible in $k$. Also, when $E$ is the constant \etale sheaf
$\Z_X$, $CH_i(X,\Z_X)$ is the usual Chow group $CH_iX$. That makes it quite
natural to consider non-torsion coefficients.
See also Conjecture \ref{cyclemodule} for a more general
attempt to define Chow groups twisted by any birational motivic sheaf.

\begin{remark}
Define an \etale sheaf $E$ of abelian groups on a scheme $X$ to be {\it naively
locally constant }if there is
an \etale covering $\{ X_{\alpha}\to X\}$ on which $E$
is constant. For convenience, define an \etale sheaf
to be {\it locally constant }if it is a direct limit of naively locally
constant sheaves. For $X$ connected with a choice of a geometric base point,
locally constant
sheaves in this sense correspond to discrete abelian groups $M$
with a continuous action
of $\pi_1^{\et}X$, whereas naively locally constant sheaves correspond
to the special case where $M$ is fixed by
some open subgroup of $\pi_1^{\et}X$.
\end{remark}

Rost's arguments imply essentially all the formal properties
of Chow groups with twisted coefficients, although we need
an extra argument (Corollary \ref{residue}) to avoid inverting
the exponential characteristic $e$ of $k$. (By definition, $e=1$
if $k$ has characteristic zero, and $e=p$ if $k$ has characteristic
$p>0$.)
Let us see how this works. The first step is to observe
that a locally constant \etale sheaf $E$
on a scheme $X$ of finite
type over $k$ determines
a cycle module $H^*[E]$ on $X$. To describe what this means, first define
a {\it field over }$X$ to be
a field $F$ with
a morphism $\Spec F\to X$ such that $F$ is finitely generated
over $k$. Then a {\it cycle module }$M$
on $X$ is a $\Z$-graded abelian group $M(F)$
associated to every field over $X$,
along with various operations on these groups. In our case,
given a locally constant \etale sheaf $E$ on $X$, we define
$H^*[E](F)$ as the \etale cohomology groups:
$$H^*[E](F)=\oplus_{j\geq 0} H^j(F,E(j)).$$
Here $\Z(j)$ denotes the \'etale sheafification of Voevodsky's
motivic cohomology complex, and $E(j)=E\otimes_{\Z}^L \Z(j)$. In particular,
if $E$ is torsion of exponent $r$ invertible in $k$, then
we have the more elementary description $E(j)\cong E\otimes_{\Z/r}
\mu_r^{\otimes j}$ in $D(X_{\et})$, giving the description
of the cycle module $H^*[E]$ from Rost's paper
\cite[Remark 1.11]{Rost}.

For clarity, we recall the operations required of a cycle module.
First, for each inclusion $\varphi\colon F_1\to F_2$ of fields over $X$,
we are given a homomorphism $\varphi_*\colon M(F_1)\to M(F_2)$ of degree 0
(a ``pullback'' homomorphism, in geometric language). For each
finite extension $\varphi\colon F_1\to F_2$ of fields over $X$, we have
a ``transfer'' or ``pushforward''
homomorphism $\varphi^*\colon M(F_2)\to M(F_1)$ of degree 0.
For each field $F$ over $X$,
the group $M(F)$ is a graded left module over the Milnor
$K$-theory ring $K_*^MF$. Finally, suppose that a field $F$ has a
``valuation over $X$'', meaning a discrete valuation $v$
together with a morphism
$\Spec O_v\to X$ such that $O_v$ is the local ring of a normal variety
over $X$ at a point of codimension 1.
Then we are given a ``residue'' homomorphism $\partial_v\colon
M(F)\to M(k(v))$ of degree $-1$, where $k(v)$ is the residue
field of $v$. We omit the relations that these operations are required
to satisfy, for $M$ to be called
a cycle module \cite[definitions 1.2 and 2.1]{Rost}.

We need the \etale sheaf $E$ on $X$ to be locally constant
in order to define the residue homomorphisms on $H^*[E]$.
Namely,
for a discrete valuation $v$ over $X$ on a field $F$ with residue field
$k(v)$, $E$ pulls back
to a locally constant \etale sheaf on $\Spec O_v$, and that is the situation
in which we have a residue homomorphism on \etale motivic cohomology:
$$\partial_v\colon H^b(F,E(b))\to H^{b-1}(k(v),E(b-1)).$$
This is easier to construct if $E$ is a $\Z[1/e]$-module; then
it comes from an isomorphism in the derived category of \etale
sheaves on $\Spec k(v)$:
$$\Z[1/e](b-1)_{k(v)}[-2]\cong i^{!}\Z[1/e](b)_{O_v}$$
for any $b\geq 1$, where $i\colon \Spec k(v)\to \Spec O_v$
is the inclusion and $i^!$ is the exceptional inverse image functor
\cite[proof of Proposition 7.1.10]{CDetale}. This fits into the localization
sequence for \etale motivic cohomology due to Cisinski and D\'eglise:
for a regular closed subscheme $Y$ of codimension $r$ in
a regular excellent scheme $X$, we have a long exact sequence
\cite[Theorem 5.6.2 and Proposition 7.1.6]{CDetale}:
\begin{multline*}
H^{i-2r}_{\et}(Y,\Z[1/e](j-r))\to H^i_{\et}(X,\Z[1/e](j))
\to H^i_{\et}(X-Y,\Z[1/e](j))\\
\to H^{i-2r+1}_{\et}(Y,\Z[1/e](j-r)).
\end{multline*}
Here $\Z[1/e](i)$ is the \etale sheafification of the usual motivic cohomology
complex (with $e$ inverted)
for $i\geq 0$, whereas it is torsion for $i<0$: $\Z[1/e](i)\cong
\oplus_{l\neq e}\Q_l/\Z_l(i)[-1]$ for $i<0$.

In characteristic $p>0$,
the residue map in the derived category of \etale sheaves
does not exist without inverting $e=p$.
Nonetheless, we construct a residue homomorphism
$$\partial_v\colon H^b(F,E(b))\to H^{b-1}(k(v),E(b-1))$$
for a locally constant \etale sheaf $E$ on $\Spec O_v$
(without inverting $p$) in Corollary \ref{residue}.
As a result, $H^*[E]$ is a cycle module,
and so Chow groups with twisted coefficients satisfy
the desired properties without having to invert
the exponential characteristic.

Once we know that $H^*[E]$ is a cycle module,
Rost's theory implies essentially all the formal properties
one would want for Chow groups with twisted coefficients. 
We have:

\begin{itemize}
\item Proper pushforward. For a proper morphism $f\colon X\to Y$
of schemes over $k$ and a locally constant \etale sheaf $E$ on $Y$,
we have a homomorphism
$$f_*\colon CH_i(X,f^*E)\to CH_i(Y,E)$$
\cite[section 5]{Rost}.
\item Flat pullback. For a flat morphism $f\colon X\to Y$
of relative dimension $n$ and a locally constant \etale sheaf $E$
on $Y$, we have a homomorphism
$$f^*\colon CH_i(Y,E)\to CH_{i+n}(X,f^*E)$$
\cite[section 5]{Rost}.
\item Localization sequence. For a closed subscheme $Z\subset X$
and a locally constant \etale sheaf $E$ on $X$, we have an exact
sequence
$$CH_i(Z,E)\to CH_i(X,E)\to CH_i(X-Z,E)\to 0.$$
This sequence can be extended to the left,
using the cycle module $H^*[E]$. Writing
$C_i(X,E)_j$ for Rost's $C_i(X;H^*[E],j)$, we define
$$C_i(X,E)_j=\oplus_{x\in X_{(i)}} H^{i+j}(k(x),E(i+j)).$$
Let $A_i(X,E)_j$ be the homology of the boundary maps on these groups,
$$C_{i+1}(X,E)_j\to C_i(X,E)_j\to C_{i-1}(X,E)_j.$$
Then the localization sequence extends to a long exact sequence
\cite[section 5]{Rost}:
\begin{multline*}
\cdots \to A_{i+1}(X-Z,E)_{-i}\\
\to A_i(Z,E)_{-i}
\to A_i(X,E)_{-i}
\to A_i(X-Z,E)_{-i}\to 0.
\end{multline*}
\item Homotopy invariance.
For an affine bundle $\pi\colon V\to X$ of relative
dimension $n$ and a locally constant \etale sheaf $E$ on $X$,
the pullback
$$\pi^*\colon CH_i(X,E)\to CH_{i+n}(V,\pi^*E)$$
is an isomorphism \cite[Proposition 8.6]{Rost}.
\item Products. For a smooth scheme $X$ over $k$, 
write $CH^i(X,E)$ for the codimension-$i$ Chow group
with coefficients. (Thus, if $X$ has dimension $n$ everywhere,
we have $CH^i(X,E)=CH_{n-i}(X,E)$.) Then $CH^*(X,E)$ is a module
over the usual Chow ring $CH^*X$ \cite[section 14]{Rost}.
\item Pullback for smooth schemes. For any morphism
$f\colon X\to Y$ of smooth schemes over $k$ and a
locally constant \etale sheaf $E$ on $Y$,
we have a homomorphism
$$f^*\colon CH^i(Y,E)\to CH^i(X,f^*E)$$
\cite[section 12]{Rost}.
\end{itemize}

Rost also proves the expected compatibilities
among these operations. For the constant \etale sheaf $A_X$ associated
to an abelian group $A$, the operations above are
the usual operations on $CH_i(X,A_X)\cong CH_i(X)\otimes_{\Z}A$.

\section{Basic calculations}

We now give some basic calculations of Chow groups
with twisted coefficients, emphasizing cases in which
they reduce to the usual Chow groups.

\begin{lemma}
\label{induced}
Let $f\colon Y\to X$ be a finite \etale morphism of schemes
of finite type over a field $k$. For a locally constant \etale
sheaf $E$ on $Y$,
$$CH_i(X,f_*E)\cong CH_i(Y,E).$$
\end{lemma}

Equivalently, Chow groups of $X$
with coefficients in an induced representation $E$
of the fundamental group reduce to Chow groups with coefficients
for a covering space of $X$. When $E$ is a permutation representation
of the \etale fundamental group
$\pi_1X$ (over some commutative ring $R$), $CH_i(X,E)$
is the usual Chow group $CH_i(Y)\otimes_{\Z} R$
of a covering space $Y$ of $X$ (possibly with several
connected components).

\begin{proof}
(Lemma \ref{induced})
One can prove this by hand, but an efficient approach is to use
Rost's results about an arbitrary morphism $f\colon Y\to X$
\cite[Corollary 8.2]{Rost}. Namely, for any cycle module $M$ on $Y$,
there is a convergent ``Leray-Serre'' spectral sequence
$$E^2_{pq} = A_p(X,A_q[f;M]) \imp A_{p+q}(Y;M),$$
for some cycle modules $A_q[f;M]$ on $X$. Namely, for each field
$F$ over $X$, let $Y_F=Y\times_X \Spec F$. Then we define
$$A_q[f;M](F) = A_q(Y_F;M).$$

Let $f\colon Y\to X$ be a finite \etale morphism, and let $E$
be a locally constant \etale sheaf on $Y$. In this case,
$Y_F$ has dimension 0
for each field $F$ over $X$. So we read off that
$$A_q[f;H^*[E]]\cong
\begin{cases} 0 &\text{if }q\neq 0\\
H^*[f_*E] &\text{if }q=0.
\end{cases} $$
Therefore, the spectral sequence reduces to an isomorphism
$A_i(X,H^*[f_*E])\cong A_i(Y,H^*[E])$ of graded abelian groups.
In degree $-i$, this gives that
$CH_i(X,f_*E)\cong CH_i(Y,E)$, as we want.
\end{proof}

Next, we consider the relation between Chow groups with twisted
coefficients and the usual Chow groups. We get a complete
answer with rational coefficients. Namely, let $G$ be a finite group,
$f\colon Y\to X$ an \etale $G$-torsor,
and $E$ a $\Z G$-module. Then we can view $E$ as a locally constant
\etale sheaf on $X$, and every sheaf associated to a representation
of $\pi_1X$ with finite image arises this way. In this situation,
we have the flat pullback homomorphism
$$CH_i(X,E)\to CH_i(Y,f^*E)=CH_i(Y)\otimes_{\Z}E.$$
Since $f^*E$ is a $G$-equivariant sheaf on $Y$, this homomorphism
lands in the $G$-fixed subgroup:
$$CH_i(X,E)\to (CH_i(Y)\otimes_{\Z}E)^G.$$
One may ask how close this is to an isomorphism; it is always
an isomorphism tensor $\Q$, as we will see.

Since $f$ is finite, we also have the pushforward homomorphism
$$CH_i(Y)\otimes_{\Z}E = CH_i(Y,f^*E)\to CH_i(X,E).$$
Using that $f^*E$ is a $G$-equivariant sheaf on $Y$,
this homomorphism factors through the coinvariants of $G$:
$$(CH_i(Y)\otimes_{\Z} E)_G\to CH_i(X,E).$$
Again, one may ask how close this is to an isomorphism.

\begin{lemma}
\label{rational}
Both homomorphisms above become isomorphisms tensor $\Q$.
\end{lemma}

\begin{proof}
Consider the composition
$$CH_i(X,E)\to (CH_i(Y)\otimes E)^G\to (CH_i(Y)\otimes E)_G
\to CH_i(X,E).$$
This is $f_*f^*$ on twisted Chow groups, which is multiplication
by $|G|$, as one can check on generators. Next, consider
the composition
$$(CH_i(Y)\otimes E)^G\to (CH_i(Y)\otimes E)_G\to CH_i(X,E)
\to (CH_i(Y)\otimes E)^G.$$
This is $f^*f_*$, which is the trace $\sum_{g\in G}g^*$. (This follows
from the fact that proper pushforward commutes with flat pullback
\cite[Proposition 4.1(3)]{Rost},
applied to the Cartesian diagram
$$\xymatrix@R-10pt{
Y\times_X Y \ar[r] \ar[d] & Y\ar[d] \\
Y \ar[r] & X,
}$$
where $Y\times_X Y\cong Y\times G$.)

Since we are considering $f^*f_*$ on the $G$-invariant subgroup
of $CH_i(Y)\otimes E$, it follows that $f^*f_*$ is multiplication
by $|G|$ on this subgroup. Thus, tensoring with $\Q$, $f^*$ gives
an isomorphism
$$CH_i(X,E)\otimes\Q \cong (CH_i(Y)\otimes E)^G\otimes \Q,$$
as we want. Also, for any abelian group $A$ with an action of the finite
group $G$,
the natural homomorphism
$A^G\to A_G$ becomes an isomorphism tensor $\Q$. It follows
that $f_*$ also gives an isomorphism tensor $\Q$:
$$(CH_i(Y)\otimes E)_G\otimes\Q\cong CH_i(X,E)\otimes\Q.$$
\end{proof}

Finally, we observe that Chow groups with twisted coefficients
do not always have the exactness properties one might wish for,
by analogy with cohomology. For example, given an exact sequence
$0\to A\to B\to C\to 0$ of locally constant \etale sheaves on $X$,
we have a long exact sequence of \etale cohomology groups,
$$\cdots\to H^i_{\et}(X,A)\to H^i_{\et}(X,B)\to H^i_{\et}(X,C)
\to H^{i+1}_{\et}(X,A)\to\cdots.$$
(When $k=\C$, we also have the analogous sequence of ordinary
cohomology groups with twisted coefficients.)
For twisted Chow groups, we will give a partial substitute
in Theorem \ref{coflasque}.

\begin{lemma}
\label{negative}
Let $0\to A\to B\to C\to 0$ be a short exact sequence
of locally constant \etale sheaves on a scheme $X$
of finite type over $k$.
Then the complex $0\to CH_i(X,A)\to CH_i(X,B)\to CH_i(X,C)\to 0$
need not be exact at any of the three terms. There are counterexamples
with $X$ smooth over $k$.
\end{lemma}

\begin{proof}
Let $G$ be the group $\Z/2$. The regular representation
of $G$ over $\F_2$, $B=\F_2 G$,
fits into a short exact sequence $0\to A\to B\to C\to 0$
of $G$-modules,
with both $A$ and $C$ isomorphic to the trivial representation $\F_2$.
We will give the desired counterexamples for this coefficient
sequence.

Let $f\colon Y\to X$ be a finite \etale morphism of degree 2.
Let $B=f_*(\F_2)_Y$.
Then the exact sequence above gives an exact sequence $0\to A
\to B\to C\to 0$ of locally
constant \etale sheaves on $X$, with both $A$ and $C$ isomorphic
to $(\F_2)_X$. By Lemma \ref{induced}, the resulting complex
of Chow groups with twisted coefficients has the form:
$$0\to CH_i(X)/2\to CH_i(Y)/2\to CH_i(X)/2\to 0,$$
where the first homomorphism is pullback and the second
is pushforward.

For example, take $k=\C$, $X=A^1-0$, $Y=A^1-0$, and
define $f\colon Y\to X$ by $f(y)=y^2$.
Then exactness fails on the right for $i=1$. (The generator
of $CH_1Y=\Z$ maps to 2 times the generator of $CH_1X=\Z$, hence
to zero modulo 2.) Next, take $k=\C$, $X=(A^2-0)/G$ (where $G=\Z/2$
acts by $\pm 1$), and $Y=A^2-0$. Then $CH_1X\cong CH^1BG\cong \Z/2$,
whereas
$CH_1Y=0$, and so the sequence is not exact on the left.

For the middle, let $k=\Q$ and let $E$ be an elliptic curve over
$\Q$ such that the Mordell-Weil group $E(\Q)$ contains
$(\Z/2)^2$ and has rank at least 1. (For example, $E$ could be the curve
\cite[\href{https://www.lmfdb.org/EllipticCurve/Q/117.a3}{Elliptic Curve 117.a3}]{lmfdb}.) Let $Y$ be $E$ minus the 2-torsion subgroup $E[2]$;
so $Y$ is $E$ minus
4 rational points. Let $G=\Z/2$ act on $Y$ by $\pm 1$; then
$G$ acts freely on $Y$, and 
$X:=Y/G$ is isomorphic to $\P^1_{\Q}$ minus 4 rational points.
So $CH_0X=0$ and hence $CH_0(X)/2=0$. On the other hand,
$CH_0Y\cong E(\Q)/E[2]$, and so $CH_0(Y)/2\neq 0$. Thus the sequence
$CH_0(X)/2\to CH_0(Y)/2\to CH_0(X)/2$ is not exact.
\end{proof}

\begin{remark}
We can also give examples over $\C$ for which exactness
fails in the middle, in Lemma \ref{negative}. Let $M$ be the K3 surface
which is the double cover of $\P^2_{\C}$
ramified along the smooth sextic curve
$C=\{0=x^6+y^6+z^6-10(x^3y^3+y^3z^3+z^3x^3)\}$. Mukai observed that
the automorphism group of $M$ contains the Mathieu group
$M_9=3^2 Q_8$ of order 72,
one of the largest finite groups of symplectic automorphisms
of a K3 surface \cite[Theorem 0.3]{Mukai}.
His arguments imply that $M$ has Picard group $\Z^{20}$.
Let $X=\P^2-C$, which
has the \etale double cover $Y:=M-C$. Then $CH^1X\cong \Z/6$,
and so $CH^1(X)/2\cong \Z/2$. On the other hand,
$C$ has genus 10, and so $C^2=2g-2=18$ on the K3 surface $M$.
It follows that $C$ is not divisible
by 2 in $CH^1M$: if $C\sim 2D$, then $C^2=4D^2$, but $C^2$ is not zero
modulo 4. So $CH^1(Y)/2=(CH^1(M)/2)/\langle C\rangle$
is isomorphic to $(\Z/2)^{19}$.
It follows that the sequence
$$CH^1(X)/2\to CH^1(Y)/2\to CH^1(X)/2$$
is not exact.
\end{remark}

\section{Chow groups and coflasque resolutions}

We now give a sufficient condition for an exact sequence
of coefficient modules to give an exact sequence of twisted
Chow groups. The statement uses the notion of a coflasque
resolution from the theory of algebraic tori.

Let $G$ be a finite group and $M$ a $\Z G$-lattice, meaning
a finitely generated $\Z G$-module that is $\Z$-projective.
Following Colliot-Th\'el\`ene and Sansuc, $M$ is called
{\it invertible }if it is a summand of a permutation module
\cite[section 0.5]{CTSflasque}.
Next, $M$ is {\it coflasque }if $H^1(H,M)=0$ for every subgroup
$H$ of $G$. Finally, $M$ is {\it flasque }if the dual lattice
$M^*$ is coflasque. An invertible $\Z G$-lattice is flasque and coflasque.
By Endo and Miyata, every coflasque $\Z G$-lattice
is invertible if and only if every Sylow subgroup of $G$ is cyclic
\cite[Proposition 2]{CTSequiv}.

More generally, let $R$ be a Dedekind domain
that is $\Z$-torsion free. We can then make the same definitions
for an $RG$-lattice, meaning a finitely generated $RG$-module
that is $R$-projective. For example, for a prime number $p$,
the localization $R=\Z_{(p)}$ or the completion $R=\Z_p$
come up naturally.

For a profinite group $L$ (such as the \etale
fundamental group of a scheme), we define an $R$-module $M$
with continuous $L$-action to be invertible, coflasque,
or flasque if there is a finite quotient group $G$ of $L$
such that $M$ is a $R G$-lattice with the corresponding property.

For a finite group $G$, every finitely generated $RG$-module $M$
has a {\it coflasque resolution}
$$0\to Q\to P\to M\to 0,$$
meaning that $P$ is a permutation module over $R$ and $Q$ is coflasque
\cite[Lemma 0.6]{CTSflasque}.
Moreover, $Q$ is determined by $M$ up to direct sums
with permutation modules.

\begin{theorem}
\label{coflasque}
Let $X$ be a $k$-scheme of finite type, and let $0\to A\to B\to C\to 0$
be an exact sequence of locally constant \etale sheaves on $X$. Let $i$ be
an integer.

(1) If $A$ is coflasque, then $CH_i(X,B)\to CH_i(X,C)\to 0$
is exact.

(2) If $A$ is invertible, then $CH_i(X,A)\to CH_i(X,B)
\to CH_i(X,C)\to 0$ is exact.
\end{theorem}

\begin{example}
In some cases, Theorem \ref{coflasque} describes Chow groups
with twisted coefficients in terms of the usual Chow groups
of varieties. For example, let $G$ be a finite group,
and let $M=\Z G/\Z$. (For example, if $G=\Z/2$, then $M$ is $\Z$
with $G$ acting by $\pm 1$.) Let $X$ be the quotient of a $k$-scheme
$Y$ by a free $G$-action. (One could assume that $Y$ is quasi-projective
to ensure that $X$ is a scheme, or use Remark \ref{space}.)
Then applying Theorem \ref{coflasque}
to the coflasque resolution $0\to \Z\to \Z G\to M\to 0$
gives an exact sequence
$$CH_iX\to CH_iY\to CH_i(X,M)\to 0,$$
where the first homomorphism is pullback. This describes $CH_i(X,M)$
in terms of Chow groups of varieties.

More generally, for any finite group $G$, every finitely generated
$\Z G$-module $M$ has a resolution
$0\to Q\to P\to M\to 0$ with $P$ a permutation module and $Q$
coflasque. By Theorem \ref{coflasque}, given a homomorphism
$\pi_1X\to G$, $CH_i(X,M)$ is always
a quotient of the usual Chow group $CH_i$ of some covering space of $X$
(possibly with several connected components). (This also follows
from Theorem \ref{transfer}, below.) When $G$ is cyclic,
$Q$ is invertible by Endo--Miyata's result above; in that case,
Theorem \ref{coflasque} expresses $CH_i(X,M)$ more explicitly
as a cokernel of a homomorphism between usual Chow groups.
\end{example}

\begin{remark}
In Theorem \ref{coflasque}, if $A$ is coflasque (but not invertible),
$CH_i(X,A)\to CH_i(X,B)
\to CH_i(X,C)$ need not be exact (Theorem \ref{counterexample}).
\end{remark}

\begin{proof}
(Theorem \ref{coflasque})
Consider the diagram of \etale cohomology groups, with exact columns:
$$\xymatrix@R-10pt{
\oplus_{x\in X_{(i+1)}} H^1(k(x),A(1)) \ar[r]\ar[d] 
&\oplus_{x\in X_{(i)}} H^0(k(x),A) \ar[d]\\
\oplus_{x\in X_{(i+1)}} H^1(k(x),B(1)) \ar[r]\ar[d]  
&\oplus_{x\in X_{(i)}} H^0(k(x),B) \ar[d]\\
\oplus_{x\in X_{(i+1)}} H^1(k(x),C(1)) \ar[r]\ar[d]  
  &\oplus_{x\in X_{(i)}} H^0(k(x),C) \ar[d]\\
\oplus_{x\in X_{(i+1)}} H^2(k(x),A(1)) \ar[r]
  &\oplus_{x\in X_{(i)}} H^1(k(x),A).
}$$
The cokernels of the first three horizontal maps are
$CH_i(X,A)$, $CH_i(X,B)$, and $CH_i(X,C)$.

Proof of (1): Suppose that $A$ is coflasque. Then $H^1(k(x),A)=0$
for every point $x$ in $X$. It follows that $H^0(k(x),B)\to H^0(k(x),C)$
is surjective for each point $x$ in $X$.
Therefore, $CH_i(X,B)\to CH_i(X,C)$
is surjective.

Proof of (2): Suppose that $A$ is invertible. Then $A$ is coflasque,
and so (1) gives that $CH_i(X,B)\to CH_i(X,C)$ is surjective.
Furthermore, since $A$ is $\Z$-torsion free, we have
$H^2(k(x),A(1))\cong H^1(k(x),A\otimes_{\Z} G_m)$ for every
point $x$ in $X$. For $R=\Z$, that group is $H^1$ with coefficients
in an algebraic torus over $k(x)$. Since $A$ is invertible,
this $H^1$ group is zero, using Hilbert's Theorem 90
that $H^1(F,G_m)=0$ for every field $F$.

Then a diagram chase implies that $CH_i(X,A)\to CH_i(X,B)
\to CH_i(X,C)$ is exact. In more detail, let $u$ be an element
of $CH_i(X,B)$ that maps to zero in $CH_i(X,C)$. Choose a representative
for $u$ in $Z_i(X,B):=\oplus_{x\in X_{(i)}} H^0(k(x),B)$. Then the image of $u$
in $Z_i(X,C)$ is the boundary of some element
$y$ in $\oplus_{x\in X_{(i+1)}} H^1(k(x),C(1))$. Since $H^2(k(x),A(1))=0$
by the previous paragraph, $y$ comes from some element $z$
in $\oplus_{x\in X_{(i+1)}} H^1(k(x),B(1))$. Then $u-\partial z$
in $Z_i(X,B)$ maps to zero in $Z_i(X,C)$. Since the columns
in the diagram above are exact, $u-\partial z$ comes from an element
of $Z_i(X,A)$, as we want.
\end{proof}

\section{Twisted motivic cohomology}
\label{motivic}

In this section we define twisted motivic cohomology associated
to a locally constant \etale sheaf, following Heller--Voineagu--\O stv\ae r
\cite[section 5.2]{HVO}. We show that twisted motivic cohomology
$H^{2i}_{\M}(X,E(i))$ can be described as a twisted
higher Chow group $CH^i(X,E,0)$ (Corollary \ref{higher}).
A striking point is that this
is not always isomorphic to the twisted Chow group
$CH^i(X,E)$ from section \ref{definition} (Theorem \ref{counterexample}).
Both theories deserve to be investigated;
we compare their advantages in Remark \ref{comparison}.

Heller, Voineagu, and \O stv\ae r
consider an action of a finite group $G$ on a smooth
scheme $Y$ over a field $k$, with the order of $G$ invertible in $k$.
They define {\it Bredon motivic
cohomology }with coefficients
in a cohomological Mackey functor $M$ for $G$, $H^i_G(Y,M(j))$. In this paper,
we only consider the case where $G$ acts freely on $Y$; equivalently,
we are considering invariants of $X:=Y/G$ with its given $G$-torsor.
Then most of the information in the Mackey functor
is irrelevant.
Every $\Z G$-module $E$ determines a cohomological Mackey functor $M$
by setting $M(G/H)=E^H$. In this case, Heller--Voineagu--\O stv\ae r's
theory coincides with twisted motivic cohomology $H^i_{\M}(X,E(j))$,
as defined below.

We can imitate Hoyois's simple definition of motivic cohomology
and the construction by Kahn--Levine and Elmanto--Nardin--Yakerson
of motivic cohomology
twisted by an Azumaya algebra \cite{Hoyoislocalization,
KL, ENY}. Namely,
let $X$ be an arbitrary scheme,
and let $E$ be a locally constant \etale sheaf on $X$. Then $E$
determines a presheaf (also called $E$)
on the category $\Sm_X$ of smooth schemes over $X$,
taking a smooth morphism $\pi\colon Y\to X$ to $H^0(Y,\pi^*E)$.
Since this is an \etale sheaf, it is a Nisnevich sheaf on $\Sm_X$.
(The Nisnevich topology is defined for arbirary schemes
in \cite[Appendix C]{HoyoisLefschetz}.)
Also, $E$ has transfers for finite locally free morphisms in $\Sm_X$
(cf.\ \cite[Lemma 6.11]{MVW}). A fortiori, $E$ has framed transfers
(that is, transfers for finite syntomic morphisms with a trivialization
of the cotangent complex). As such, $E$ defines a space
$E_X$ in the framed motivic homotopy category $\HH^{\fr}(X)$.
So $E$ defines
a motivic spectrum $HE_X:=\Sigma^{\infty}_{\T,\fr}E_X$ in the stable
homotopy category $\SH(X)$. (This is not the suspension spectrum
of a motivic space. Rather, a {\it framed }motivic space is analogous
to an $E_{\infty}$ space in topology, and $\Sigma^{\infty}_{\T,\fr}$ denotes
the left adjoint to the functor $\Omega^{\infty,\fr}_{\T}$ from $\SH(X)$
to $\HH^{\fr}(X)$ (not to $\HH(X)$).) For the constant
sheaf $E=\Z_X$, Hoyois showed
that this object $H\Z_X$ coincides with Spitzweck's motivic
cohomology spectrum in $\SH(X)$, for every scheme $X$
\cite[Theorem 21]{Hoyoislocalization}.

In particular, this gives a definition of
$E$-twisted motivic cohomology:
$$H^i_{\M}(X,E(j))=\pi_{2j-i}\map_{\SH(X)}(\Sigma^{\infty}_{\T}X_+,
\Sigma^j_{\T}HE_X).$$
This is analogous to the definition of motivic cohomology twisted
by an Azumaya algebra \cite[Definition 5.17]{ENY}. In particular,
this agrees with Spitzweck's definition of motivic cohomology when $E$
is a constant sheaf, and with Voevodsky's definition when in addition
$X$ is defined over a field.

\begin{lemma}
\label{motexact}
Let $0\to A\to B\to C\to 0$ be a short exact sequence
of locally constant \etale sheaves on a scheme $X$.
Suppose that $A$ is coflasque.
Then, for each integer $j$, we have a long exact sequence
of twisted motivic cohomology groups:
$$\cdots\to H^i_{\M}(X,A(j))\to H^i_{\M}(X,B(j))\to H^i_{\M}(X,C(j))
\to H^{i+1}_{\M}(X,A(j))\to\cdots$$
\end{lemma}

\begin{remark}
Using this lemma, we can always describe $H^{2i}_{\M}(X,C(i))$
as the cokernel of an explicit map between the usual Chow groups
of covering spaces of $X$ (not necessarily connected).
Namely, assume that $C$ is a finitely generated
$\Z$-module with an action of $\pi_1^{\et} X$, and let $0\to A\to B
\to C\to 0$ be a coflasque resolution. We have $H^{2i+1}_{\M}(X,A(i))=0$
by the relation to twisted higher Chow groups (Corollary \ref{higher}
below), and so Lemma \ref{motexact} expresses $H^{2i}_{\M}(X,C(i))$
as the cokernel of $H^{2i}_{\M}(X,A(i))\to H^{2i}_{\M}(X,B(i))$.
By taking a coflasque resolution of $A$, we can describe
$H^{2i}_{\M}(X,A(i))$ as the image of a homomorphism
from $H^{2i}_{\M}(X,P(i))$,
for some permutation module $P$. Since $P$ and $B$ are permutation
modules, we have expressed $H^{2i}_{\M}(X,C(i))$ as the cokernel
of a map between the usual Chow groups of covering spaces of $X$.

The twisted Chow groups can likewise be expressed as the image
of a homomorphism from
the usual Chow groups of some covering space by Theorem \ref{coflasque}
(or by the surjective homomorphism $H^{2i}_{\M}(X,C(i))\to
CH^i(X,C)$ from Theorem \ref{surjective}, below). I don't know
a simple description of the kernel, though.
\end{remark}

\begin{proof}
For any exact sequence $0\to A\to B\to C\to 0$ of locally constant
\etale sheaves on $X$, we have an exact sequence
$0\to A_X\to B_X\to C_X$ of Nisnevich sheaves with transfer.
Suppose in addition that $A$ is coflasque.
Let $Y$ be a smooth scheme over $X$,
and let $L$ be the henselization of the local ring
of $Y$ at a point $x$ in $Y$. (We are not taking
the {\it strict }henselization. So
the residue field of $L$ is $k(x)$, not the separable closure
of $k(x)$.) I claim that $H^1_{\et}(L,A)=0$. Indeed, we have 
$H^1_{\et}(L,A)\cong H^1(k(x),A)$.
Let $G$ be the image
of $\pi_1(\Spec L)=\Gal(k(x)_s/k(x))$ acting on $A$, which is a finite group.
Let $H$ be the kernel of $\Gal(k(x)_s/k(x))\to G$.
Then the Lyndon--Hochschild--Serre spectral sequence gives
an exact sequence
$$0\to H^1(G,A)\to H^1(k(x),A)\to H^1(H,A)^G.$$
Here $H^1(G,A)$ is zero since $A$ is coflasque,
and $H^1(H,A)=\Hom(H,A)$ is zero since $H$ is profinite
and $A$ is a discrete torsion-free abelian group.
So $H^1_{\et}(L,A)=H^1(k(x),A)=0$, as claimed.

For each henselian local ring $L$ of a scheme $Y$ in $\Sm_X$,
we have an exact sequence
of \etale cohomology:
$$0\to A(L)\to B(L)\to C(L)\to H^1_{\et}(L,A).$$
Therefore, the previous paragraph gives that $0\to A(L)\to
B(L)\to C(L)\to 0$ is exact. That is, $0\to A_X\to B_X\to C_X\to 0$
is an exact sequence of Nisnevich sheaves on $X$ \cite[p.~90]{MVW}.
As a result, we get an exact triangle $HA_X\to HB_X\to HC_X$ in $\SH(X)$.
That implies the desired long exact sequence.
\end{proof}

We can also define twisted higher Chow groups. These groups should agree
with twisted motivic cohomology for a regular scheme.
We prove this in bidegrees $(2j,j)$,
and in any bidegree when $X$ is smooth over a perfect field $k$
and the \etale sheaf $E$ is pulled back from $k$
(Lemma \ref{overfield} and Corollary \ref{higher}).
Namely, for an equidimensional scheme $X$ with a locally constant
\etale sheaf $E$,
define the $j$th {\it twisted
Bloch cycle complex }$z^j(X,E)$ as the simplicial abelian group
(or the associated chain complex)
which in degree $d$ is given by
$$\oplus_{z\in X\times \Delta^d} H^0(k(z),E),$$
where the sum is over all codimension-$j$ points of 
$X\times \Delta^d$ whose closure intersects all faces of $X\times \Delta^d$
in the expected dimension. (Here $\Delta^d$ denotes the algebraic simplex
$\{ x_0+\cdots+x_d=0\}$ in $A^{d+1}$, as in Bloch's definition
of higher Chow groups.) The boundary maps in the chain complex
are given by intersections with faces of $X\times \Delta^d$.
Write $CH^j(X,E,d)$ for the homotopy group $\pi_d$ of the simplicial
abelian group $z^j(X,E)$ (or, equivalently, $H_d$ of the associated
chain complex).

\begin{lemma}
\label{overfield}
Let $X$ be an equidimensional smooth scheme of finite type
over a perfect field $k$. Let $E$ be an \etale sheaf over $k$,
pulled back to $X$.
Then twisted motivic cohomology agrees with twisted higher Chow groups:
$$H^i_{\M}(X,E(j))\cong CH^j(X,E,2j-i).$$
\end{lemma}

\begin{proof}
We follow the proof of the analogous isomorphism for motivic
cohomology twisted by an Azumaya algebra \cite[Proposition 5.15]{ENY}.
The key point is that the $A^1$-invariant Nisnevich sheaf $E_k$
on $\Sm_k$ is birational,
in the sense that $E_k(Y)\cong E_k(U)$ for a dense open subset $U$
of $Y$ in $\Sm_k$. (This follows from the fact that $Y$ is normal
\cite[Tag 0BQI]{Stacks}.) In terms of Voevodsky's slice filtration,
it follows that $HE_k$ is a 0-slice in $\SH(k)$, meaning that $HE_k
\simeq s_0 HE_k$ \cite[Proposition 5.3]{ENY}. We can rewrite the spectrum
used to define the twisted motivic cohomology of $X$ in terms
of the morphism $f\colon X\to \Spec k$:
\begin{align*}
\map_{SH(X)}(\Sigma^{\infty}_{\T}X_+,\Sigma^j_{\T}HE_X)
&\simeq \map_{SH(X)}(\Sigma^{\infty}_{\T}X_+,f^*(\Sigma^j_{\T}HE_k))\\
&\simeq \map_{SH(k)}(Lf_{\#} (\Sigma^{\infty}_{\T}X_+),\Sigma^j_{\T}HE_k)\\
&=\map_{SH(k)}(\Sigma^{\infty}_{\T}X_+,\Sigma^j_{\T}HE_k),
\end{align*}
where in the last line we write $\Sigma^{\infty}_{\T}X_+$ for 
the object of $\SH(k)$ associated to $X$.
By Levine's results
on the homotopy coniveau tower, using the smoothness
of $X$ over $k$, it follows that this mapping spectrum in $\SH(k)$
is equivalent to the simplicial spectrum
$$\oplus_{z\in X\times \Delta^{\bullet}}(s_0HE_k)(k(z)),$$
where the sum is indexed by every codimension-$j$ point of $X\times
\Delta^{\bullet}$ whose closure meets all faces in the expected
dimension \cite[Corollary 5.3.2 and Theorem 9.0.3]{Levine}.
Since $s_0HE_k\simeq HE_k$, the latter spectrum is equivalent
to the twisted Bloch cycle complex.
\end{proof}

\begin{corollary}
\label{higher}
Let $X$ be an equidimensional smooth scheme of finite type
over a perfect field $k$. Let $E$ be a locally constant \etale
sheaf on $X$ such that the exponential characteristic of $k$
acts invertibly on $E$.
Then twisted motivic cohomology in degree $(2j,j)$
agrees with twisted higher Chow groups:
$$H^{2j}_{\M}(X,E(j))\cong CH^j(X,E,0).$$
\end{corollary}

\begin{proof}
Let $X^a$ be the ``union of all subvarieties of codimension at least $i$
in $X$''. Precisely,
the twisted motivic cohomology of $X-X^a$ means the direct
limit of the motivic twisted cohomology of $X-S$ over all closed subsets
of codimension at least $a$ in $X$. Then, by taking a direct limit
of localization sequences, we have an exact sequence for any $i,j,a$:
\begin{multline*}
\cdots \to\oplus_{x\in X^{(a)}}H^{i-2a}_{\M}(k(x),E(j-a))
\to H^{i}_{\M}(X-X^{a+1},E(j)) \\
\to H^i_{\M}(X-X^a,E(j))
\to\oplus_{x\in X^{(a)}}H^{i-2a+1}_{\M}(k(x),E(j-a))\to\cdots.
\end{multline*}
Here $X^{(a)}$ denotes the set of points whose closure has codimension $a$
in $X$.

The invariants for fields (on the left and right) can be identified
with twisted higher Chow groups, by Lemma \ref{overfield}.
Many of these groups are trivially zero. In particular,
the groups contributing to $H^{2i}_{\M}(X,E(i))$
are $CH^{i-b}(k(x),E,0)$ for points $x$ of codimension $b$ in $X$,
and this group is zero for $b\neq i$. Explicitly, the exact sequence above
for $a=i$ gives:
$$H^{2i-1}_{\M}(X-X^i,E(i))\to\oplus_{x\in X^{(i)}}CH^0(k(x),E,0)
\to H^{2i}_{\M}(X,E(i))\to 0.$$

To clarify this, we also want generators for $H^{2i-1}_{\M}(X-X^i,E(i))$.
The groups contributing to this are $CH^{i-b}(k(x),E,1)$
for points $x$ of codimension $b$ in $X$ with $0\leq b\leq i-1$.
This is zero for $b\neq i-1$. For $a=i-1$, the exact sequence above
shows that $\oplus_{x\in X^{(i-1)}}CH^1(k(x),E,1)$ maps onto
$H^{2i-1}_{\M}(X-X^i,E(i))$. Thus we have an exact sequence:
$$\oplus_{x\in X^{(i-1)}}CH^1(k(x),E,1)
\to\oplus_{x\in X^{(i)}}CH^0(k(x),E,0)
\to H^{2i}_{\M}(X,E(i))\to 0.$$

For a field $F$ with an \etale sheaf $E$,
we have $CH^0(F,E,0)\cong H^0(F,E)$. Likewise,
$CH^1(F,E,1)$ is generated by $H^0(L,E)$ for closed points
$\Spec L$ in $(\Spec F)\times \Delta^1$. These are the same
as the generators and relations for $CH^i(X,E,0)$, and so we have shown
that
$$H^{2i}_{\M}(X,E(i))\cong CH^i(X,E,0).$$
\end{proof}

\begin{corollary}
\label{surjective}
Let $X$ be a smooth variety over a field $k$ and $E$ a locally
constant \etale sheaf on $X$. Assume that the exponential
characteristic of $k$ acts invertibly on $E$.
Then there is a natural surjection
$$H^{2i}_{\M}(X,E(i))\to CH^i(X,E).$$
\end{corollary}

\begin{proof}
By Corollary \ref{higher}, $H^{2i}_{\M}(X,E(i))$ is isomorphic
to the twisted higher Chow group $CH^i(X,E,0)$. The generators
for this group are the same as for $CH^i(X,E)$, and so the natural
homomorphism $CH^i(X,E,0)\to CH^i(X,E)$ is surjective.
\end{proof}

We will see that this surjection is not always an isomorphism
(Theorem \ref{counterexample}). To describe the difference
in the definitions: the relations in $CH^i(X,E)$
are $H^1(k(x),E(1))$ for points $x$ of codimension $i-1$
in $X$, while the relations in $CH^i(X,E,0)$
are the transfers to $H^1(k(x),E(1))$ of products
of $H^0(F,E)$ with $H^1(F,\Z(1))=F^*$, for finite extensions
$F$ of $k(x)$.

\section{Equivariant Chow groups with coefficients}
\label{equivariant}

Following the definition of equivariant Chow groups \cite{TotaroChow,
EG, Totarobook}, Guillot, Di Lorenzo, and Pirisi
defined equivariant Chow groups with coefficients in a cycle module
\cite{Guillot, DLP}. We focus here on the special case
of equivariant Chow groups with
coefficients in a $\Z G$-module $M$, $CH^G_*(X,M)$. In particular,
that gives a definition
of the Chow groups of the classifying space with coefficients in $M$,
$CH^*(BG_k,M)=CH^*_G(\Spec k,M)$.

Namely, let $X$ be a scheme of finite type over a field $k$
with an action of a finite group $G$. Let $M$ be a $\Z G$-module.
Let $i$ be an integer.
Let $V$ be any representation of $G$ over $k$ such that $G$
acts freely on a Zariski-closed subset $S\subset V$
with $\codim(S\subset V)>\dim(X)-i$.
Then we define the $i$th {\it equivariant Chow group }with coefficients
in $M$ by:
$$CH^G_i(X,M)=CH_{i+\dim(V)}((X\times (V-S))/G,M).$$
By the same proof as for $CH_i^G(X)$ (using homotopy invariance
and the localization sequence), this group is independent
of the choice of $(V,S)$.

\begin{remark}
\label{space}
If $X$ is quasi-projective
over $k$, then $(X\times (V-S))/G$ is also a quasi-projective scheme, and so
its twisted Chow groups have been defined in section \ref{definition}.
If $X$ is not quasi-projective, then $(X\times (V-S))/G$ is (in general)
only an algebraic space. In that case,
one has to use that the definition
of twisted Chow groups works without change for algebraic spaces,
as in \cite[section 6.1]{EG}.
\end{remark}

When $X$ is smooth and equidimensional over $k$,
define
$$CH^i_G(X,M)=CH^G_{\dim(X)-i}(X,M).$$
(If $X$ is smooth but has components
of different dimensions, define $CH^i_G(X,M)$ to be the direct sum
of these groups for each $G$-orbit of components.)

As in the references above,
the formal properties of equivariant Chow groups with coefficients
follow immediately from the properties of Chow groups with
coefficients. In particular, we have proper pushforward, flat pullback,
homotopy invariance, the localization sequence, and (in the smooth case)
products. Namely, for smooth $k$-schemes, $CH^*_G(X,M)$ pulls back
under arbitrary $G$-equivariant morphisms,
and $CH^*_G(X,M)$ is a module over the ring $CH^*_GX$.

\section{The \etale cycle map}

We now construct the cycle map from twisted Chow groups
to \etale motivic cohomology. For smooth varieties, we define the cycle map
in full generality, without having to invert the exponential
characteristic of $k$.
In that generality, the construction is subtle:
\etale motivic cohomology need not satisfy the localization sequence
in the usual form, and we use instead a new purity result,
Corollary \ref{Ecodim}.
For singular varieties, the cycle map takes
values in \etale Borel-Moore motivic homology, which we only consider
with the exponential characteristic inverted.

\begin{theorem}
\label{etalecycle}
Let $X$ be a scheme of finite type over a field $k$,
and let $M$ be a locally constant \etale
sheaf on $X$.

(1) Suppose that $X$ is regular.
Then we define a natural homomorphism
$$CH^r(X,M)\to H^{2r}_{\et}(X,M(r)).$$

(2) Without assuming that $X$ is regular, suppose that the exponential
characteristic of $k$ acts invertibly on $M$.
Then we define a natural homomorphism
$$CH_j(X,M)\to H_{2j,\et}^{BM}(X,M(j))$$
to \etale motivic Borel-Moore homology. 
\end{theorem}

\begin{remark}
In Theorem \ref{etalecycle}(2), we use a version of \etale motivic
Borel-Moore homology defined by Cisinski and D\'eglise. They also suggested
that the ``right'' definition
of \etale Borel-Moore motivic homology,
especially if we do not want to invert
the exponential characteristic of $k$, 
would be \etale
cohomology with coefficients in Bloch's higher Chow complexes
\cite[Remark 7.1.4]{CDetale}. In more detail, we should define
$$H_{i,\et}^{BM}(X,\Z(j))=H^{-i}_{\et}(X,D_X(-j)),$$
where $D_X(-j)$ is a shift of Bloch's cycle complex, numbered
by dimension of cycles. Namely,
for $U$ \etale over $X$, $D_U(-j)=z_{j,U}[2j]$, where $z_{j,U}$
is the cochain complex
in which $(z_{j,U})^i$ is the free abelian group on the set
of $(j-i)$-dimensional subvarieties
in $X\times \Delta^{-i}$ that meet all faces
in the expected dimension.

For a locally constant \etale sheaf $M$, we can therefore define
$$H_{i,\et}^{BM}(X,M(j))=H^{-i}_{\et}(X,M\otimes_{\Z}^L D_X(-j)).$$
I conjecture that
there is a cycle map
$$CH_i(X,M)\to H_{2i,\et}^{BM}(X,M(i)),$$
without having to invert the exponential characteristic of $k$.
\end{remark}

\begin{proof}
(1) We first define the cycle map on generators.
Let $Y$ be a subvariety of codimension $r$ in $X$,
and let $u$ be an element of $H^0(k(Y),E)$.
Write $i\colon Y\to X$ for the inclusion. First suppose
that $Y$ is regular. Then $H^0_{\et}(Y,E)\cong H^0(k(Y),E)$,
and so we can view $u$ as an element of $H^0_{\et}(Y,E)$
\cite[Tag 0BQI]{Stacks}.
Then the morphism $E[-2r]\to i^!E(r)$ in the derived category $D_{\et}(Y)$
(Corollary \ref{Ecodim}) gives a homomorphism
$H^0_{\et}(Y,E)\to H^0_{\et}(Y,i^!E(r)[2r])\cong H^{2r}_{\et}(Y,i^!E(r))
\cong H^{2r}_{Y,\et}(X,E(r))$. We can map further to $H^{2r}_{\et}(X,E(r))$,
as we want.

Next, let $Y$ be a subvariety of codimension $r$ in $X$,
possibly singular, and let $u$ be an element of $H^0(k(Y),E)$.
Let $Y'$ be the singular locus of $Y$, which has codimension
at least $r+1$ in $X$. The previous paragraph defines an element
of $H^{2r}_{Y-Y'}(X-Y',E(r))$, and so it suffices to show
that $H^{2r}_Y(X,E(r))\to H^{2r}_{Y-Y'}(X-Y',E(r))$ is an isomorphism.
Consider the exact sequence:
$$H^{2r}_{Y'}(X,E(r))\to H^{2r}_Y(X,E(r))\to
H^{2r}_{Y-Y'}(X-Y',E(r))\to H^{2r+1}_{Y'}(X,E(r)).$$
So it suffices to show that for a closed subset $Y'$ of codimension
everywhere at least $s$ in $X$ and any $a<s$, $H^{2a}_{Y'}(X,E(a))=
H^{2a+1}_{Y'}(X,E(a))$. This follows from Corollary \ref{Ecodim},
which says in terms of $f\colon Y'\inj X$ that
$\tau_{\leq 2a+1}f^!E(a)=0$.

Next, let us show that this map on generators passes
to a well-defined homomorphism $CH^r(X,E)\to H^{2r}_{\et}(X,E(r))$.
A relation is given by a closed subvariety $Y$ of codimension $r-1$
in $X$ and an element $u\in H^1(k(Y),E(1))$. We want to show
that $\partial u\in Z^r(X,E)$ maps to zero
in $H^{2r}_{\et}(X,E(r))$. More precisely, we will show
that $u$ maps to zero in $H^{2r}_{Y,\et}(X,E(r))$.
Write $i\colon Y\inj X$ for
the inclusion. 

Here $u$ comes from an element
of $H^1_{\et}(Y-D,E(1))$ for some reduced divisor $D$ in $Y$,
say $D=D_1+\cdots+D_s$. Let $D_{\sing}$ be the singular locus of $D$,
so $D-D_{\sing}=\coprod_{j=1}^s D_j^0$ for some regular codimension-1
subvarieties $D_j^0$ of $Y-D_{\sing}$. Since $Y-D$ is a regular subvariety
of the regular scheme $X-D$, the morphism $E(1)[-2r+2]\to i^!E(2r)$
in $D_{\et}(Y-D)$ (Corollary \ref{Ecodim}) gives a Gysin
homomorphism
$H^1_{\et}(Y-D,E(1))\to H^{2r-1}_{Y-D}(X-D,E(r))$. So we can view $u$
as an element of the latter group. Consider the exact sequence
of \etale cohomology groups:
$$H^{2r-1}_{Y-D}(X-D,E(r))\to \oplus_{j=1}^s H^{2r}_{D_j^0}(X-D_{\sing},
E(r))\to H^{2r}_{Y-D_{\sing}}(X-D_{\sing},E(r)).$$
Since $D_j^0$ is a regular subscheme of the regular scheme $X-D_{\sing}$,
we have the purity isomorphism $H^0(D_j^0,E)\cong
H^{2j}_{D_j^0}(X-D_{\sing},E(r))$ (Corollary \ref{Ecodim}).
So the exact sequence above shows that $\partial u$
in $Z^r(X,E)$ maps to zero in $H^{2r}_{Y-D_{\sing}}(X-D_{\sing},E(r))$.
Finally, the restriction map $H^{2r}_Y(X,E(r))\to
H^{2r}_{Y-D_{\sing}}(X-D_{\sing},E(r))$ is an isomorphism, since
$D_{\sing}$ has codimension at least $r+1$ in $X$ (by Corollary
\ref{Ecodim} again). So the image of $\partial u$
in $H^{2r}_Y(X,E(r))$ is zero, as we want. Thus the cycle map
$CH^i(X,E)\to H^{2i}_{\et}(X,E(r))$ is well-defined.

(2) Now allow $X$ to be singular, but assume
that the exponential characteristic $e$ of $k$
is invertible in $M$. In this case,
Cisinski and D\'eglise defined one version
of the \etale motivic Borel-Moore homology
of $X$ over $k$
\cite[Remark 7.1.12(4)]{CDetale}. Write $f\colon X\to \Spec k$.
Namely, they set
$$H_i^{BM}(X,\Z[1/e](j))=H^{-i}_{\et}(X,B_X(-j))$$
for an object $B_X(-j)$ in $D_{\et}(X)$
(the object $f^{!}1_k(-j)$ in the category of $h$-motives
$DM_h(X,\Z[1/e])$, which has a functor $R\alpha_*$ to $D_{\et}(X)$).
Therefore, we can define twisted Borel-Moore homology by
$$H_i^{BM}(X,M(j))=H^{-i}_{\et}(X,M\otimes_{\Z}^L B_X(-j)),$$
where we assume that $M$ is a locally constant \etale sheaf
on which $e$ acts invertibly.

We have
$H^i_{\et}(X,M(j))\cong H_{2n-i,\et}^{BM}(X,M(n-j))$ for $X$ smooth over $k$,
everywhere of dimension $n$. Also, for a closed subscheme
$S$ of a scheme $X$ over $k$, we have a localization exact sequence:
\begin{multline*}
\cdots\to H_{i,\et}^{BM}(S,M(j))\to H_{i,\et}^{BM}(X,M(j))\\
\to H_{i,\et}^{BM}(X-S,M(j))
\to H_{i-1,\et}^{BM}(S,M(j))\to\cdots.
\end{multline*}

Let $X_a$ be the ``union of all subvarieties of dimension at most $a$
in $X$'', as in the proof of Corollary \ref{higher}.
By taking a direct limit of localization sequences, we have
an exact sequence for any $i,j,a$:
\begin{multline*}
\cdots\to \oplus_{x\in X_{(a)}}H^{2a-i}_{\et}(k(x),M(a-j))
\to H_{i,\et}^{BM}(X-X_{a-1},M(j))\\
\to H_{i,\et}^{BM}(X-X_a,M(j))
\to \oplus_{x\in X_{(a)}}H^{2a-i+1}_{\et}(k(x),M(a-j))\to\cdots.
\end{multline*}

We first define the cycle map
$CH_b(X,M)\to H_{2b,\et}^{BM}(X,M(b))$ on generators.
So let $x$ be a point in $X$ whose closure has dimension $b$,
and let $u$ be an element of $H^0(k(x),M)$. By the localization sequence
above, this maps to an element $u$ in $H_{2b,\et}^{BM}(X-X_{b-1},M(b))$.
For each $0\leq c\leq b-1$, the restriction map
$$H_{2b,\et}^{BM}(X-X_{c-1},M(b))\to H_{2b,\et}^{BM}(X-X_{c},M(b))$$
is an isomorphism by the localization sequence, using that
$H^{i}_{\et}(F,M(j))=0$ for fields $F$ when $i$ and $j$ are negative.
(Indeed, the complex
of \etale sheaves $\Z[1/e](j)$
with $j<0$ is concentrated
in cohomological degree 1 by section \ref{definition},
and so $M(j)=M\otimes_{\Z}^L\Z[1/e](j)$
is concentrated in degrees $\geq 0$.) Therefore, $u$ comes from a unique
element of $H_{2b,\et}^{BM}(X,M(b))$.

It remains to show that the cycle map vanishes on the relations defining
$CH_b(X,M)$. So let $x$ be a point in $X$ whose closure
has dimension $b+1$, and let $u$ be an element of $H^1(k(x),M(1))$.
We want to show that the boundary $\partial u$ in $Z_b(X,M)$ maps to zero
in $H_{2b,\et}^{BM}(X,M(b))$. It suffices to prove this with $X$ replaced
by the closure of the point $x$;
that is, we can assume that $X$ is a variety
of dimension $b+1$ over $k$. As above, we have the localization sequence
$$H_{2b+1,\et}^{BM}(X-X_b,M(b))
\to \oplus_{x\in X_{(b)}}H^{0}_{\et}(k(x),M)
\to H_{2b,\et}^{BM}(X-X_{b-1},M(b)).$$
Since $X$ has dimension $b+1$, the first group here
is $H^1_{\et}(k(X),M(1))$. So the sequence shows that $\partial u$
in $\oplus_{x\in X_{(b)}}H^{0}_{\et}(k(x),M)$
maps to zero in $H_{2b,\et}^{BM}(X-X_{b-1},M(b))$. The previous paragraph
defines the image of $\partial u$ in the finer group
$H_{2b,\et}^{BM}(X,M(b))$. But we showed in the previous paragraph
that the restriction from 
$H_{2b,\et}^{BM}(X,M(b))$ to $H_{2b,\et}^{BM}(X-X_{b-1},M(b))$
is an isomorphism. So $\partial u$ is zero in $H_{2b,\et}^{BM}(X,M(b))$.
Thus we have a well-defined homomorphism $CH_b(X,M)\to
H_{2b,\et}^{BM}(X,M(b))$.
\end{proof}

\section{The cycle map for complex varieties}

\begin{theorem}
Let $X$ be a scheme of finite type over $\C$, $M$ a locally constant \etale
sheaf on $X$, and $i$ an integer. Then we define a natural homomorphism
$$CH_i(X,M)\to H_{2i}^{BM}(X,M)$$
to Borel-Moore homology (using the classical topology). For $X$ smooth
over $\C$, we can rephrase this as a homomorphism
$$CH^j(X,M)\to H^{2j}(X,M).$$
\end{theorem}

\begin{proof}
Let $X$ be a locally compact space.
To define Borel-Moore homology with twisted coefficients, recall
that $H_i^{BM}(X,\Z)$ can be described as $H^{-i}(X,D_X)$, where
$D_X$ is the dualizing complex and we consider
cohomology in the classical topology \cite[Equation IX.4.1]{Iversen}.
For a locally constant \etale sheaf $M$ on $X$,
we can define $H_i^{BM}(X,M)=H^{-i}(X,M\otimes^L_{\Z}D_X)$.

Let $X$ be a scheme of finite type over $\C$.
We first define the cycle map on generators. Let $y$ be a point
of the scheme $X$
whose closure $Y$ has dimension $i$,
and let $u$ be an element of $H^0(k(y),M)$. Let $S$ be the singular set
of $Y$. Since $S$ has dimension at most $i-1$ in $X$,
the localization sequence
$$\to H_{2i}^{BM}(S,M)\to H_{2i}^{BM}(X,M)\to H_{2i}^{BM}(X-S,M)
\to H_{2i-1}^{BM}(S,M)\to $$
shows that $H_{2i}^{BM}(X,M)\to
H_{2i}^{BM}(X-S,M)$ is an isomorphism. So it suffices to define an element
of $H_{2i}^{BM}(X-S,M)$ associated to $u$. Replacing $X$ by $X-S$,
we can assume that $Y$ is smooth. Then $H^0(Y,M)\cong
H^0(k(y),M)$ since $Y$ is normal
\cite[Tag 0BQI]{Stacks}.
So we can view $u$ as an element of $H^0(Y,M)
\cong H_{2i}^{BM}(Y,M)$. Proper pushforward
gives a homomorphism
$H_{2i}^{BM}(Y,M)\to H_{2i}^{BM}(X,M)$.
So $u$ gives an element of $H_{2i}^{BM}(X,M)$, as we want.

It remains to show that for a point $w$ of the scheme $X$
whose closure has dimension $i+1$ in $X$ and an element $t$
of $H^0(k(w),M(1))$, the boundary of $t$ maps to zero in $H_{2i}^{BM}(X,M)$.
Because Galois cohomology commutes with direct limits, $t$
comes from a cohomology class in some Galois submodule
of $M$ that is finitely generated as a $\Z$-module.
So we can assume that $M$ is finitely generated as a $\Z$-module, and we want
to show that $\partial t$ maps to zero in $H_{2i}^{BM}(X,M)$.
We can assume that $X$ is connected. In this case,
in terms of the classical topology, $M$ is a module
for the fundamental group $\pi_1X$ that factors through some finite quotient
group $G$ of $\pi_1X$.

Since the dualizing complex $D_X$ is constructible,
so is $M\otimes^L_{\Z}D_X$,
and hence the group $H_{2i}^{BM}(X,M)\cong H^{-2i}(X,M\otimes^L_{\Z}D_X)$
is finitely generated. So if we can show
that $\partial t$ in $H_{2i}^{BM}(X,M)$ is divisible, then it is zero,
as we want. The image of $\partial t$ is zero in $H_{2i}^{BM}(X,M/n)$
for every positive integer $n$, because we know that $\partial t$
maps to zero in $H_{2i,\et}^{BM}(X,M(i))$ (Theorem \ref{etalecycle}),
hence in $H_{2i,\et}^{BM}(X,M/n(i))$,
which (as we are over $\C$) can be identified
with $H_{2i}^{BM}(X,M/n)$. By the exact sequence
$$H_{2i}^{BM}(X,nM)\to H_{2i}^{BM}(X,M)\to H_{2i}^{BM}(X,M/n),$$
it follows that $\partial t$ lies in the image 
of $H_{2i}^{BM}(X,nM)$ for every positive integer $n$.
If $M$ is $\Z$-torsion-free
(so $nM\cong M$), then
it follows that $\partial t$ is divisible, hence zero as we want.

In general, let $a$ be the order of the torsion subgroup of $M$.
Then $aM$ is a torsion-free $\Z G$-submodule of $M$. We know that
$\partial t$ in $H_{2i}^{BM}(X,M)$ is in the image of $H_{2i}^{BM}(X,naM)$
for every positive integer $n$. So for every positive integer $n$,
$\partial t$ is the image of an element of $H_{2i}^{BM}(X,aM)$
that is divisible by $n$, and hence $\partial t$ itself is divisible by $n$.
Thus $\partial t=0$ in $H_{2i}^{BM}(X,M)$, as we want. Finally,
for $X$ smooth over $k$, the cycle map we have constructed can be
rewritten as $CH^j(X,M)\to H^{2j}(X,M)$, via Poincar\'e duality.
\end{proof}

\section{Twisted Chow groups and transfers}

We now observe that twisted Chow groups
have explicit generators in terms of transfers.
This generalizes Merkurjev--Scavia's description
of the negligible subgroup of group cohomology in degree 2
\cite[Theorem 1.3, Corollary 4.2]{MS}.
It is also related to Theorem \ref{coflasque},
since it shows again that twisted Chow groups
are a quotient of the usual Chow groups of a suitable
covering space of $X$ (possibly with several connected
components).

\begin{theorem}
\label{transfer}
Let $X$ be a $k$-scheme of finite type, $G$ a finite group,
$Y\to X$ a $G$-torsor, and $M$ a $\Z G$-module. Then $CH_i(X,M)$
is generated by the images of the homomorphisms
$$M^H\otimes_{\Z} CH_i(Y/H)\to CH^0(Y/H,M)\otimes_{\Z} CH_i(Y/H)
\to CH_i(Y/H,M)\to CH_i(X,M)$$
over all subgroups $H$ of $G$,
where the last map is the transfer or pushforward.

More strongly, one does not need to use all subgroups of $G$.
For each element $x\in M$, let $G_x$ be the centralizer
of $x$ in $G$. Then $CH_i(X,M)$ is generated by the elements
$\tr_{G_x}^G(xy)$ for all $x\in M$ and all $y\in CH_i(Y/G_x)$.
\end{theorem}

\begin{proof}
The group $CH_i(X,M)$ is generated by the groups $H^0(k(z),M)$
for all points $z$ in $X$ with closure $Z$ of dimension $i$.
Given such a point, let $H$ be the image of the composition
$\pi_1(\Spec k(z))\to \pi_1X\to G$. Then the restriction of the covering map
$f\colon Y/H\to X$ over $z$ has a section, $z_1\in Y/H$. Let $Z_1$
be its closure in $Y/H$ (which maps birationally to $Z$).
Let $u$ be any element of $H^0(k(z),M)=M^H$. Then the element
of $CH_i(X,M)$ associated to the pair $(z,u)$
is the pushforward of the product $u[z_1]$
in $M^H\otimes_{\Z}CH_i(Y/H)\to CH^0(Y/H,M)\otimes_{\Z} CH_i(Y/H)
\to CH_i(Y/H,M)$.

That proves the first statement, that $CH_i(X,M)$ is generated by
$\tr_{H}^G(xy)$ for all subgroups $H$ in $G$,
$x\in M^H$ and $y\in CH_i(Y/H)$. Here $H$ is contained
in the centralizer $G_x$ of $x$ in $G$. By the projection formula,
\begin{align*}
\tr_H^G(xy)&=\tr_{G_x}^G \tr_H^{G_x}(\res^{G_x}_H(x) y)\\
&=\tr_{G_x}^G(x\tr_H^{G_x}y).
\end{align*}
Thus $CH_i(X,M)$ is generated by the elements $\tr_{G_x}^G(xz)$
for all $x\in M$ and all $z\in CH_i(Y/G_x)$.
\end{proof}

To deduce Merkurjev--Scavia's statement, we need the following
simple observation.

\begin{lemma}
\label{support}
Let $X$ be a smooth $k$-variety, $G$ a finite group,
$Y\to X$ a $G$-torsor, and $M$ a $\Z G$-module.
Assume that the exponential characteristic of $k$
acts invertibly on $M$.
Then, for $i>0$, the cycle class homomorphism $CH^i(X,M)
\to H^{2i}_{\et}(X,M(i))$
maps into the kernel of restriction to $H^{2i}(k(X),M(i))$.
For $i=1$, the image of $CH^1(X,M)\to H^2_{\et}(X,M(1))$ is equal
to the kernel of restriction to $H^{2}(k(X),M(1))$.
\end{lemma}

\begin{proof}
For $i>0$, it is clear that the homomorphism $CH^i(X,M)\to H^{2i}_{\et}(X,M(i))$
maps into the kernel of restriction to $H^{2i}(k(X),M(i))$, because it gives
classes supported on a codimension-$i$ subset of $X$.

Let $i=1$. Let $u$ be an element of $H^{2}_{\et}(X,M(1))$ that restricts
to zero in $H^{2}(k(X),M(1))$. Then there is a closed subset $S$
of codimension at least 1 in $X$ such that $u$ restricts
to zero on $X-S$. By the localization sequence for \etale motivic
cohomology (section \ref{definition}),
$H^2_{\et}(X,M(1))$ does not change after removing
a subset of codimension at least 2 from $X$. After doing so,
we can assume that $S$ is a disjoint union of regular codimension-1
subvarieties of $X$.

For a regular codimension-1 subvariety $Z$ of $X$, the localization
sequence for \etale motivic cohomology has the form
$$H^0_{\et}(Z,M)\to H^2_{\et}(X,M(1))\to H^2_{\et}(X-Z,M(1)).$$
By definition, elements of $H^0_{\et}(Z,M)=H^0_{\et}(k(Z),M)$
are in the image of $CH^1(X,M)$. Thus we have shown that the kernel
of restriction to the generic point is contained
in the image of $CH^1(X,M)$; so it is equal to that image.
\end{proof}

\begin{corollary}
\label{ms}
Let $G$ be a finite group and $M$ a finite $G$-module. Let $k$
be a separably closed field in which $|G|$ and $|M|$
are invertible. Then the subgroup of elements of
$H^2(G,M)$ that are negligible over $k$ is generated by the images
of the maps
$$M^H\otimes H^2(H,\Z)\to H^2(H,M)\xrightarrow[\tr_H^G]{} H^2(G,M),$$
where $H$ runs over all subgroups of $G$.

More strongly, one does not need to use all subgroups of $G$.
For each element $x\in M$, let $G_x$ be the centralizer
of $x$ in $G$. Then the negligible subgroup of $H^2(G,M)$
is generated by the elements
$\tr_{G_x}^G(xy)$ for all $x\in M$ and all $y\in H^2(G_x,\Z)$.
\end{corollary}

\begin{proof}
By definition, an element $u$ of $H^i(G,M)$ is called {\it negligible
over $k$ }if for every field $K$ containing $k$ and every continuous
homomorphism
$\Gal(K_s/K)\to G$, the induced homomorphism $H^i(G,M)\to H^i(K,M)$ takes
$u$ to zero. Let $V$ be a faithful representation of $G$; then the resulting
$G$-torsor over $k(V/G)$ is versal in Serre's sense. In particular,
an element $u$ of $H^i(G,M)$ is negligible over $k$
if and only if it pulls back to zero in $H^i(k(V/G),M)$
\cite[Proposition 2.1, Corollary 2.2]{GM}.

Let $U$ be an open subset of $V$ such that $G$ acts freely
on $U$. For a suitable choice of representation $V$, we can assume
that $V-U$ has codimension at least 2 in $V$. Then $CH^1(U/G,M)
\cong CH^1(BG_k,M)$ by definition of the latter group. Likewise,
$H^2_{\et}(U/G,M)\cong H^2(G,M)$ by the Hochschild--Serre
spectral sequence, using that $k$ is separably closed
\cite[Theorem III.2.20]{Milne}.
By Theorem \ref{transfer},
an element $u$ in $H^2(G,M)$
is negligible if and only if, as an element
of $H^2(G,M)\cong H^2_{\et}(U/G,M)\cong H^2_{\et}(U/G,M(1))$,
it is in the image of $CH^1(U/G,M)\cong CH^1(BG_k,M)$.
Note that $CH^1(BH_k)\cong H^2(H,\Z)\cong \Hom(H,k^*)$ for each
subgroup $H$ of $G$ \cite[Lemma 2.26]{Totarobook}. Therefore,
Theorem \ref{transfer} and Lemma \ref{support} yield
the two statements we want.
\end{proof}

Corollary \ref{ms} is Merkurjev--Scavia's result,
specialized to the case where
the base field $k$ is separably closed
\cite[Theorem 1.3, Corollary 4.2]{MS}. Their result for an arbitrary
base field suggests the following question,
which follows from their results when $i=1$:

\begin{question}
\label{roots}
Let $G$ be a finite group, $M$ a $\Z G$-module which is killed
by some positive integer. Let $e(G)$ be the exponent of $G$
and $r:=e(M)$ the exponent of $M$.
Let $k$ be a field such that
$e(G)e(M)$ is invertible in $k$ and $k$ contains the $e(G)e(M)$
roots of unity. Let $i\geq 0$. Does the cycle map
$$CH^i(BG_k,M)\to H^{2i}_{\et}(BG_k,M(i))$$
factor through the homomorphism
$$H^{2i}(G,M\otimes_{\Z/r} \mu_r(k)^{\otimes i})
\to H^{2i}_{\et}(BG_k,M(i))?$$
\end{question}

By Theorem \ref{transfer}, Question \ref{roots} would follow
from the special case $M=\Z/r$ (with trivial action of $G$).
By Gherman and Merkurjev, in the case $i=1$, the assumption
that $k$ contains the $e(G)e(M)$ roots of unity is sharp
\cite[Theorem 4.2]{GM}. Earlier, Grothendieck had considered
Question \ref{roots} in the special case of Chern
classes of representations (with $M=\Z/r$),
although without the precise
hypothesis on the $e(G)e(M)$ roots of unity
\cite[section 5]{Grothendieckclasses}.

\section{Twisted Chow groups of a cyclic group}

In this section,
for a finite cyclic group $G$,
we compute the Chow groups of $BG$ with arbitrary coefficients.
In this case, the twisted Chow groups are periodic and have
a simple relation to group cohomology.

\begin{theorem}
\label{cyclic}
Let $G$ be the cyclic group of order $m$, and let $M$ be a finitely generated
$\Z G$-module. Let $k$ be a field such that $m$ is invertible in $k$
and $k$ contains the $m$th roots of unity. Then
$$CH^i(BG_k,M)\cong \begin{cases} M^G &\text{if }i=0\\
M^G/\tr(M) &\text{if }i>0,
\end{cases}$$
where the trace $\tr$ is $\sum_{g\in G}g$.
For $k=\C$, we can also say that the natural homomorphism
$$CH^i(BG_{\C},M)\to H^{2i}(BG,M)$$
is an isomorphism for all $i$. (The right side is usually
written as $H^{2i}(G,M)$, group cohomology with coefficients.)
\end{theorem}

Thus $CH^*(BG_{\C},M)\to H^{\ev}(BG,M)$
is an isomorphism for $G$ cyclic. Note that the cohomology
may be nonzero in odd degrees for $G$ cyclic. Namely, for $i\geq 1$ odd,
$H^i(G,M)\cong \ker(\tr\colon M\to M)/\im(1-\sigma)$,
where $G=\langle\sigma:\sigma^m=1\rangle\cong \Z/m$ and $\tr=1+\sigma+\cdots
+\sigma^{m-1}$. For example, for $G=\Z/2$, $H^1(G,M)$ is nonzero
for $M=\F_2$, or for $M=\Z G/\Z$ (which is $\Z$ with $G$ acting by $\pm 1$).

\begin{proof}
(Theorem \ref{cyclic})
Since $k$ contains the $m$th roots of unity, $G$ has a faithful
representation $V$ of dimension 1. Clearly $G$ acts freely on $V-0$,
and so we can apply the following lemma.

\begin{lemma}
\label{euler}
Let $G$ be a finite group. Suppose that $G$ has a representation $V$
of dimension $n>0$
over a field $k$ such that $G$ acts freely on $V-0$.
Let $M$ be a finitely generated $\Z G$-module. Then multiplication
by the Euler class $c_nV$ on $CH^i(BG_k,M)$ is surjective
for $i\geq 0$ and an isomorphism
for $i\geq 1$.
\end{lemma}

\begin{proof}
Use the localization sequence for equivariant Chow groups
with coefficients, applied to the inclusion $\{ 0\}\subset V$:
$$CH^{i}_G(\{0\},M)\to CH^{i+n}_G(V,M)\to CH^{i+n}_G(V-0,M)\to 0.$$
By homotopy invariance of twisted Chow groups, the first two
groups can be identified with $CH^{i}(BG,M)$ and $CH^{i+n}(BG,M)$,
and the homomorphism is multiplication by the Euler class of $V$,
$c_nV\in CH^nBG$. Also, the third group in the exact sequence
is $CH^{i+n}((V-0)/G,M)$, because $G$ acts freely on $V-0$.
So the exact sequence says that $CH^*((V-0)/G,M)\cong
CH^*(BG,M)/(c_n(V)CH^*(BG,M)).$

We have $CH^i((V-0)/G,M)=0$ for $i>n$. More subtly, it is also zero
for $i=n$. By Theorem \ref{transfer}, it suffices to show
that $CH_0((V-0)/H)$ is zero for every subgroup $H$ of $G$.
The point is that $H$ commutes with the action of the multiplicative
group $G_m$ on $V$
by scalars, and so we have an action of $G_m$ on $(V-0)/H$ with
finite stabilizer groups. It follows that $CH_0((V-0)/H)=0$
\cite[Lemma 5.3]{Totarobook}. We conclude
that $CH^i((V-0)/G,M)=0$ for $i\geq n$.
Therefore, multiplication
by $c_nV$ is surjective on $CH^i(BG_k,M)$ for $i\geq 0$.

To prove the injectivity statement, use the previous
term in the localization sequence. In Rost's notation:
\begin{multline*}
A_{1-i}((V-0)/G,M)_{i}\\
\to A_{-i}^G(\{ 0\},M)_{i}
\to A_{-i}^G(V,M)_{i}
\to A_{-i}((V-0)/G,M)_{i}\to 0.
\end{multline*}
The first group is zero for $i>1$,
and so multiplication by $c_nV$ is injective on $CH^i(BG,M)$
for $i>1$. 

In fact, this injectivity extends to the case $i=1$.
The point is that the group $A_0((V-0)/G,M)_{1}$
may not be zero, but its boundary map to $A_{-1}^G(\{ 0\},M)_{1}$
is zero. By definition of equivariant Chow groups, for an approximation
$U/G$ to $BG$ of dimension $N$, this is identified with the boundary map
from $A_N(((V-0)\times U)/G,M)_{1-N}$ to $A_{N-1}((\{ 0\}\times U)/G,
M)_{1-N}$. But the inverse image of $((V-0)\times U)/G$ of each closed
point in $(V-0)/G$ has closure in $(V\times U)/G$ disjoint from
$(\{ 0\}\times U)/G$, and so the boundary map is zero. Thus
multiplication by $c_nV$ is injective on $CH^i(BG,M)$
for $i\geq 1$ (not just for $i>1$). Lemma \ref{euler} is proved.
\end{proof}

We can now prove Theorem \ref{cyclic}. By definition, $CH^i(BG_k,M)$
means $CH^i(U/G,M)$ for any Zariski open subset $U$
of a finite-dimensional
representation $W$ of $G$ over $k$ such that $G$ acts freely
on $U$ and $W-U$ has codimension $>i$ in $W$. In particular,
$CH^0(BG,M)\cong CH^0(U/G,M)\cong H^0(k(U/G),M)\cong M^G$, as we want.

Applying Lemma \ref{euler}
to the 1-dimensional faithful representation $V$ of $G$
shows that $CH^*(BG,M)$ is generated by $CH^0(BG,M)=M^G$
as a module over $\Z[c_1V]$. For each positive integer $i$, let $U/G$
be an approximation to $BG$ as above, and write $f\colon U\to U/G$
for the covering. Then $CH^i(BG,M)=CH^i(U/G,M)$ is generated
by $c_1(V)^i M^G$. Moreover, $f^*(c_1(V)^i)=0$ in $CH^i(U)=0$,
and so $0=f_*(f^*(c_1(V)^i) x)=c_1(V)^i\tr(x)$ for every
$x$ in $CH^0(U,f^*M)=M$. That is, $M^G/\tr(M)$ maps onto $CH^i(BG,M)$
for $i>0$.

Let $k=\C$. Then, for $i>0$, the surjection $M^G/\tr(M)\to CH^i(BG_{\C},M)$
must be an isomorphism, by mapping further to $H^{2i}(BG,M)$.
Indeed, it is a standard calculation in group cohomology that for a cyclic
group $G$, the map $M^G/\tr(M)\to H^{2i}(BG,M)$ (given by multiplication
by $c_1(V)^i$) is an isomorphism for all $i>0$.

For a general field $k$ under our assumptions ($m$ invertible in $k$
and $k$ containing the $m$th roots of unity), let us show that
the surjection $M^G/\tr(M)\to CH^i(BG_k,M)$ for $i>0$ is also injective.
Since both sides commute with direct limits, we can assume that
$M$ is a finitely generated $\Z G$-module. Also,
it suffices to prove this injectivity
after replacing $k$ by its separable closure.
For $i>0$ and each prime number $l$ dividing $m=|G|$, it suffices to show
that the homomorphism $(M^G/\tr(M))_{(l)}\to CH^i(BG_k,M)_{(l)}
\to H^{2i}_{\et}(BG_k,M^{\wedge l}(i))$ (continuous
or pro-\etale cohomology) is injective. By the Hochschild--Serre
spectral sequence, using that $k$ is separably closed,
we have $H^{2i}_{\et}(BG_k,M^{\wedge l}(i))\cong H^{2i}(G,M^{\wedge l}\otimes
\Z_l(i))$
\cite[Theorem III.2.20]{Milne}. Since $G$ is cyclic,
this is isomorphic to $(M^G/\tr(M))_{(l)}$, as we want.
\end{proof}

\section{Twisted Chow groups in codimension 1}

We now show that the codimension-1 Chow group with any twist
injects into \etale motivic cohomology. This fails in higher codimension,
even for $CH^2(X,\Z/r)=CH^2(X)/r$ (for example,
see \cite{Schoen, Totaromod2}). It also fails for twisted motivic
cohomology in codimension 1 (Theorem \ref{counterexample}).

\begin{theorem}
\label{injective}
Let $X$ be a smooth scheme of finite type over a field $k$.
Let $M$ be a locally constant \etale sheaf on $X$. Assume that
the exponential characteristic of $k$ acts invertibly on $M$.
Then the cycle map
$CH^i(X,M)\to H^{2i}_{\et}(X,M(i))$ is an isomorphism for $i=0$
and injective for $i=1$.
\end{theorem}

\begin{proof}
We can assume that $X$ is connected, hence irreducible.
By definition, we have $CH^0(X,M)\cong H^0(k(X),M)\cong H^0_{\et}(X,M)$,
as we want. 
(This follows from the fact that $X$ is normal
\cite[Tag 0BQI]{Stacks}.)
Next, let $y$ be an element of $CH^1(X,M)$ that maps
to zero in $H^2_{\et}(X,M(1))$. Then $y$ is represented by finitely
many distinct irreducible divisors $D_1,\ldots,D_s$ on $X$ together with
an element of $H^0(k(D_j),M)$ for $j=1,\ldots,s$.
Let $S$ be the singular locus
of $D:=D_1\cup\cdots\cup D_s$; then $S$ has codimension
at least 2 in $X$. Clearly $y$ maps to zero in $H^2_{\et}(X-S,M(1))$,
and we want to show that $y$ is zero in $CH^1(X,M)\cong CH^1(X-S,M)$.
Thus, we can replace $X$ by $X-S$;
then $D_1,\ldots,D_s$ are regular and disjoint in $X$. Let $D$ be the union
of $D_1,\ldots,D_s$.

Consider the commutative diagram, with top row exact:
$$\xymatrix@R-10pt{
H^1_{\et}(X-D,M(1)) \ar[r]\ar[d] &
\oplus_j H^0_{\et}(D_j,M)\ar[r]\ar[d]^{\cong} &
H^2_{\et}(X,M(1))\\
H^1(k(X),M(1)) \ar[r] & \oplus_j H^0(k(D_j),M) &
}$$
Then $y$ in $\oplus_j H^0(k(D_j),M)$ comes from an element
of $\oplus_j H^0(D_j,M)$ that maps to zero in $H^2_{\et}(X,M(1))$.
By the top row, that element is in the image of $H^1_{\et}(X-D,M(1))$,
and so $y$ is in the image of $H^1_{\et}(k(X),M(1))$. That is,
$y=0$ in $CH^1(X,M)$, as we want.
\end{proof}

\begin{corollary}
\label{injectivecor}
(1) Let $G$ be a finite group and 
$M$ a $\Z G$-module. Then
$$CH^1(BG_{\C},M)\to H^{2}(BG_{\C},M)=H^2(G,M)$$
is injective.

(2) Let $G$ be a finite group, $k$ a field,
$M$ a finitely generated $\Z G$-module, and $l$ a prime number invertible
in $k$. 
Then the kernel of $CH^1(BG_k,M)\to H^{2}_{\et}(BG_k,M^{\wedge l}(i))$
(continuous
or pro-\etale cohomology) is torsion of order prime to $l$.
\end{corollary}

\begin{proof}
In proving (1), we write $BG$ for $BG_{\C}$.
Let $y$ be an
element of $CH^1(BG,M)$ that maps to zero in $H^2(BG,M)$ (cohomology
for the classical topology); we want to show that $y$ is zero.
Since both sides commute with direct
limits, we can assume that $M$ is a finitely generated
$\Z G$-module. Since $y$ maps to zero in $H^2(BG,M)$, it maps to
zero in $H^2(BG,M/nM)\cong H^2_{\et}(BG,(M/nM)(1))$ for every positive integer
$n$. By the exact sequence
$$H^2_{\et}(BG,nM(1))\to H^2_{\et}(BG,M(1))\to H^2_{\et}(BG,(M/nM)(1)),$$
the class of $y$ in $H^2_{\et}(BG,M(1))$ is in the image
of $H^2_{\et}(BG,nM(1))$ for every positive integer $n$.
If $M$ is $\Z$-torsion-free, this means that $y$ in $H^2_{\et}(BG,M(1))$
is divisible. 
Next, observe that $H^2_{\et}(EG,M(1))\cong H^2_{\et}(\C,M(1))=0$,
since $\C$ is separably closed and $M(1)=M\otimes_{\Z}^L G_m$
is in cohomological degrees $[-1,0]$.
By pullback and pushforward for $EG\to BG$, it follows that
$H^2_{\et}(BG,M(1))$ is killed by $|G|$. In the case where $M$
is $\Z$-torsion-free, it follows that $y$ is zero in $H^2_{\et}(BG,M(1))$.
By Theorem \ref{injective}, $y$ is zero in $CH^1(BG,M)$, as we want.

In general, we have arranged that $M$ is a finitely generated
$\Z G$-module, possibly with torsion. The previous argument
shows that $y$ in $H^2_{\et}(BG,M(1))$ is in the image
of $H^2_{\et}(BG,nM(1))$ for every positive integer $n$.
Let $a$ be the order of the torsion subgroup of $M$. Then $aM$
is a $\Z$-torsion-free $\Z G$-submodule of $M$. We know 
that $y$ in $H^2_{\et}(BG,M(1))$
is in the image
of $H^2_{\et}(BG,naM(1))$ for every positive integer $n$.
Therefore, $y$ in $H^2_{\et}(BG,M(1))$ is in the image of $n$
times an element of $H^2_{\et}(BG,aM(1))$, and so $y$
is a multiple of $n$ in $H^2_{\et}(BG,M(1))$. The latter group
is killed by $|G|$, and so (taking $n=|G|$) it follows that
$y$ is zero in $H^2_{\et}(BG,M(1))$.
By Theorem \ref{injective}, $y$ is zero in $CH^1(BG,M)$, as we want.

For an arbitrary field $k$,
essentially the same argument proves (2). Here we assume
that $M$ is a finitely generated $\Z G$-module. We can replace
$M$ by $M[1/e]$ without changing what we are trying to prove, where
$e$ is the exponential characteristic of $k$.
Let $y$ be an
element of $CH^1(BG_k,M)$ that maps to zero
in $H^2_{\et}(BG_k,M^{\wedge l}(1))$. Then $y$ maps to zero
in $H^2_{\et}(BG_k,(M/l^rM)(1))$ for every positive integer $r$.
If $M$ is $\Z$-torsion-free, it follows as above that $y$
is $l$-divisible in $H^2_{\et}(BG_k,M(1))$. 

By homotopy invariance
for \etale motivic cohomology (using that $e$ acts invertibly on $M$),
the composition
$$H^2_{\et}(k,M(1))\to H^2_{\et}(BG_k,M(1))\to H^2_{\et}(EG_k,M(1))$$
is an isomorphism. Moreover, $y$ pulls back to zero in $CH^1(EG_k,M)=0$,
hence in $H^2_{\et}(EG_k,M(1))$. That is, $y$ is in the summand
$\ker(H^2_{\et}(BG_k,M(1))\to H^2_{\et}(EG_k,M(1)))$, and so $y$
is $l$-divisible in this summand. This summand
is killed by $|G|$, by pullback and pushforward.
So $y$ in $H^2_{\et}(BG_k,M(1))$
is killed by a positive integer $N$ prime to $l$. By Theorem \ref{injective},
it follows that $y$ in $CH^1(BG_k,M)$ is killed by the positive integer $N$
prime to $l$. The extra argument when $M$ has torsion works
without change.
\end{proof}

\section{The generalized quaternion groups}

The generalized quaternion group $Q_{2^m}$ of order $2^m$ with $m\geq 3$
plays a special role in finite group theory, as these
are the only non-cyclic
$p$-groups of $p$-rank 1 \cite[Proposition IV.6.6]{AM}.
We now compute the Chow groups
of the generalized quaternion groups with arbitrary coefficients.
The answer is periodic, similar to group cohomology but not identical.
Here
$$Q_{2^m}=\langle x,y:x^{2^{m-1}}=1,y^2=x^{2^{m-2}},yxy^{-1}=x^{-1}
\rangle .$$

\begin{theorem}
\label{quaternion}
Let $G$ be the generalized quaternion group $Q_{2^m}$ of order $2^m$,
$m\geq 3$, and let $M$ be a finitely generated
$\Z G$-module. Let $k$ be a field of characteristic not 2
that contains the $2^{m-1}$st roots of unity.
Then $CH^0(BG_k,M)\cong  M^G$,
\begin{multline*}
CH^i(BG_k,M)\cong \im\big( 
M^{\langle x\rangle}\otimes_{\Z} CH^1B\langle x\rangle\\
\oplus M^{\langle y\rangle}\otimes_{\Z} CH^1B\langle y\rangle
\oplus M^{\langle xy\rangle}\otimes_{\Z} CH^1B\langle xy\rangle
\to H^2_{\et}(BG_k,M^{\wedge 2}(1))\big)
\end{multline*}
if $i$ is odd and $i>0$, and
$$CH^i(BG_k,M)\cong M^G/\tr(M)$$
if $i$ is even and $i>0$,
where the trace $\tr$ is $\sum_{g\in G}g$. Here $\langle x\rangle
\cong \Z/2^{m-1}$, $\langle y\rangle\cong \Z/4$,
and $\langle xy\rangle\cong \Z/4$. The homomorphism in the formula
for $CH^1(BG_k,M)$ is the sum of the transfers from these three subgroups
to $G$, followed by the cycle map to cohomology. Finally, for $k=\C$,
we can equivalently describe $CH^1(BG_{\C},M)$ as the image
of the groups above in cohomology for the classical
topology, $H^2(BG,M)=H^2(G,M)$.
\end{theorem}

Thus the cycle map $CH^i(BG_{\C},M)\to H^{2i}(BG,M)$
is injective for a generalized quaternion group $G$. It is surjective
for $i$ even, but not always for $i$ odd. For example, for the
quaternion group $G=Q_8$, $CH^1(BG,M)$ is always killed by 4,
because it is transferred from the subgroups of order 4 in $G$,
whereas the $\Z G$-module $M:=\Omega^2\Z$
has $H^2(G,M)\cong \Z/8$. (By definition,
for a $\Z G$-module $M$, the {\it syzygy module }$\Omega M$
denotes the kernel of any chosen
surjection from a projective $\Z G$-module to $M$. Then, in terms
of Tate cohomology, $\widehat{H}^i(G,\Omega M)\cong \widehat{H}^{i-1}(G,M)$
for all integers $i$ \cite[section VI.5.4]{Brown}.)

\begin{proof}
(Theorem \ref{quaternion})
The assumption on $k$ ensures that $G=Q_{2^m}$ has a faithful
representation $V$ of dimension 2 over $k$. (Namely, $G$ has a cyclic subgroup
$H$ of index 2. Let $L$ be a faithful 1-dimensional representation
of $H$ over $k$; then we can take $V$ to be the induced representation
$\Ind_H^G L$.) Here $G$ acts freely on $V-0$.

Let $M$ be a finitely generated
$\Z G$-module. By Lemma \ref{euler}, $CH^*(BG_k,M)$
is generated by elements of degree less than 2 as a module
over $\Z[c_2V]$ in $CH^*BG$. By definition,
$CH^0(BG,M)$ is isomorphic to the $G$-fixed subgroup $M^G$.
For any finite group $H$ and any field $k$, the first Chern class
gives an isomorphism $\Hom(H,k^*)\cong CH^1BH_k$
\cite[Lemma 2.26]{Totarobook}.
So, for $G=Q_{2^m}$,
$CH^1BG\cong (\Z/2)^2$.

The main issue is to compute $CH^1(BG,M)$. As a first step,
let us show that $CH^1BG\cong (\Z/2)^2$ is generated by transfers
from the subgroups $\langle x\rangle\cong \Z/2^{m-1}$ and $\langle y
\rangle\cong \Z/4$. Let $u$ be a generator of $CH^1B\langle x\rangle$
and $v$ a generator of $CH^1B\langle y\rangle$. The double coset formula
describes the restriction of a transfer in the Chow ring,
as in group cohomology \cite[Lemma 2.15]{Totarobook}. We find that
$$\res^G_{\langle x\rangle}\tr^G_{\langle x\rangle}u=u+y(u)=u-u=0,$$
while
\begin{align*}
\res^G_{\langle y\rangle}\tr^G_{\langle x\rangle}u&=\tr_{\langle y^2\rangle}
^{\langle y\rangle}\res^{\langle x\rangle}_{\langle x^{2^{m-2}}\rangle}u\\
&=\tr_{\langle y^2\rangle}^{\langle y\rangle}(\text{generator
of }CH^1B\langle y^2
\rangle)\\
&\neq 0\in CH^1 B\langle y\rangle\cong \Z/4.
\end{align*}
It follows that $\tr^G_{\langle x\rangle}u$ is the nonzero element
of $CH^1BG\cong (\Z/2)^2$ whose restriction to $\langle x\rangle$
is zero. Next,
\begin{align*}
\res^G_{\langle x\rangle}\tr^G_{\langle y\rangle}v
&=\tr^{\langle x\rangle}_{\langle x^{2^{m-2}}\rangle}
\res^{\langle y\rangle}_{\langle y^2\rangle}v\\
&=\tr^{\langle x\rangle}_{\langle x^{2^{m-2}}\rangle}
(\text{generator of }CH^1B\langle x^{2^{m-2}}\rangle)\\
&\neq 0 \in CH^1 B\langle x\rangle\cong \Z/2^{m-1}.
\end{align*}
So $\tr^G_{\langle y\rangle}v$ in $CH^1BG\cong (\Z/2)^2$
has nonzero restriction
to $\langle x\rangle$. It follows that $CH^1BG$ is generated by
$\tr^G_{\langle x\rangle}u$ and $\tr^G_{\langle y\rangle}v$.

By the projection formula, it follows that the image
of the product map $M^G\otimes CH^1BG\to CH^1(BG,M)$
is contained in the sum of the images
of the transfers from $CH^1(B\langle x\rangle,M)$
and $CH^1(B\langle y\rangle,M)$. By Theorem \ref{transfer},
it follows that $CH^1(BG,M)$ is generated by transfers from the three
maximal subgroups of $G$, $H_1:=\langle x\rangle$, $H_2:=\langle x^2,y
\rangle$, and $H_3:=\langle x^2,xy\rangle$. Moreover, by Corollary
\ref{injectivecor}, $CH^1(BG,M)$ is isomorphic to the sum of the images
of these three
transfers in $H^2_{\et}(BG_k,M^{\wedge 2}(1))$, or (when $k=\C$)
in $H^2(BG,M)\cong H^2(G,M)$. For $m=3$, the three
subgroups $H_i$ are isomorphic to $\Z/4$, and so Theorem \ref{cyclic}
gives that $CH^1(BH_j,M)$ is generated by $M^{H_j}\otimes_{\Z}
CH^1BH_j$ for $j=1,2,3$. That gives the description of $CH^1(BG,M)$
in the theorem, for $m=3$.

For $m\geq 4$, we prove the description of $CH^1(BG,M)$ in the theorem
by induction on $m$. Here $H_1$ is isomorphic to $\Z/2^{m-1}$,
while $H_2$ and $H_3$ are isomorphic to $Q_{2^{m-1}}$. 
We know that $CH^1(BH_1,M)$ is generated by the image
of $M^{H_1}\otimes_{\Z}CH^1BH_1$, by Theorem \ref{cyclic}.
By induction on $m$, $CH^1(BH_2,M)$ is generated
by transfers from $\langle x^2\rangle\cong \Z/2^{m-2}$,
$\langle y\rangle \cong \Z/4$, and $\langle x^2y\rangle\cong \Z/4$.
Likewise, $CH^1(BH_3,M)$ is generated
by transfers from $\langle x^2\rangle\cong \Z/2^{m-2}$,
$\langle xy\rangle \cong \Z/4$, and $\langle x^3y\rangle\cong \Z/4$.
So $CH^1(BG,M)$ is generated by transfers from $\langle x\rangle$,
$\langle y\rangle$, $\langle xy\rangle$, $\langle x^2y\rangle$,
and $\langle x^3y\rangle$. However, $xyx^{-1}=x^2y$, and so
the subgroup $\langle y\rangle$ is conjugate to $\langle x^2y\rangle$,
and likewise $\langle xy\rangle$ is conjugate to $\langle x^3y\rangle$.
Therefore, $CH^1(BG,M)$ is generated by transfers from $\langle x\rangle$,
$\langle y\rangle$, and $\langle xy\rangle$. That completes the induction,
computing $CH^1(BG,M)$.

Recall that $G=Q_{2^m}$ has a representation
$V$ of dimension 2 over $k$ such that $G$ acts freely on $V-0$.
By Theorem \ref{euler}, multiplication by $(c_2V)^i$ is an isomorphism
from $CH^1(BG_k,M)$ to $CH^{2i+1}(BG_k,M)$ for $i\geq 0$, as we want.
Also, multiplication by $(c_2V)^i$ is surjective from
$CH^0(BG_k,M)=M^G$ to $CH^{2i}(BG_k,M)$ for $i>0$.
Since $c_2V$
restricts to zero in $CH^2$ of the trivial group, $\tr_1^G(M)(c_2V)^i$
is zero in $CH^{2i}(BG,M)$ for $i>0$; so $CH^{2i}(BG,M)$
is generated by $(M^G/\tr(M))(c_2V)^i$ for $i>0$.
Since $G$ has periodic cohomology with period 4, with periodicity
generated by $c_2V$, we have $H^{4i}_{\et}(BG_{k_s},M^{\wedge 2}(2i))\cong M^G/\tr(M)$
for all $i\geq 1$ and $H^{2i+2}_{\et}(BG_{k_s},M^{\wedge 2}(i+1))
\cong H^2_{\et}(BG,M^{\wedge 2}(1))$
for all $i\geq 0$ \cite[Theorem VI.9.1]{Brown}.
Via the cycle map to these cohomology groups over
the separable closure $k_s$, it follows
that $CH^{2i}(BG,M)$ is isomorphic to $M^G/\tr(M)$ for all $i\geq 1$,
as we want.
\end{proof}

\section{Bounds on equivariant Chow groups with coefficients}

Using Theorem \ref{transfer}, some earlier bounds on the degrees
of equivariant Chow groups generalize to arbitrary coefficient modules.

\begin{theorem}
\label{degreebound}
(1) Let $G$ be a finite group with a faithful representation $V$
of dimension $n$ over a field $k$ with $|G|$ invertible in $k$.
Let $M$ be a $\Z G$-module. Then $CH^*(BG_k,M)$ is generated
by elements of at most $n(n-1)/2$ as a module over the Chern
classes $\Z[c_1V,\ldots,c_nV]$ in $CH^*BG_k$.

(2) Under the same assumptions, let $x_1,\ldots,x_m$ be homogeneous elements
of $CH^*BG_k$ such that $c_1V,\ldots,c_nV$ are in the subring
of $CH^*BG_k$ generated by $x_1,\ldots,x_m$. Then $CH^*(BG_k,M)$
is generated by elements of degree at most $\sum(|x_i|-1)$
as a module over $\Z[x_1,\ldots,x_m]$.

(3) Suppose in addition that $G$ acts on a smooth scheme $X$
of finite type
over $k$. Then $CH^*_G(X,M)$ is generated by elements of degree
at most $\dim(X)+n^2$ as a module over $\Z[c_1V,\ldots,c_nV]$.
\end{theorem}

\begin{remark}
\label{nobound}
Theorem \ref{degreebound} marks a strong contrast between Chow groups
with coefficients and cohomology with coefficients.
For a finite group $G$ with a faithful complex representation $V$
of dimension $n$, Symonds
showed that $H^*(G,\F_p)$ is generated by elements of degree
at most $n^2$ as a module
over $\F_p[c_1V,\ldots,c_nV]$ \cite{Symonds},
\cite[Corollary 4.3]{Totarobook}. But unlike what happens
for twisted Chow groups, there is no uniform bound for the degrees
of module generators of $H^*(G,M)$ for all $\F_p G$-modules $M$.
For example, for $G=\Z/2\times \Z/2$ and an $\F_2 G$-module $M$,
define the syzygy module
$\Omega M$ as the kernel
of the surjection from a projective cover of $M$
to $M$. Then,
for $m\geq 0$ and $M=\Omega^m \F_2$,
$H^*(G,M)$ needs a generator of degree $m$ as a module over $H^*(G,\F_2)$,
by Benson and Carlson's results on products in Tate cohomology
\cite[Lemma 2.1 and Theorem 3.1]{BC}. See the proof of Theorem
\ref{klein} for more details.
\end{remark}

\begin{proof}
(Theorem \ref{degreebound})
Statement (1) is known for $M=\Z$ \cite[Theorem 5.2]{Totarobook}.
(Moreover, the bound $n(n-1)/2$ is optimal
\cite[section 5.2]{Totarobook}.) The same statement
applies to all subgroups $H$ of $G$, as modules over the same ring
$\Z[c_1V,\ldots,c_nV]$. The transfer from $CH^i(BH,M)$
to $CH^i(BG,M)$ is linear over $CH^*BG$, hence over $\Z[c_1V,
\ldots,c_nV]$. Therefore, Theorem \ref{transfer} gives
that $CH^*(BG,M)$ is generated 
by elements of at most $n(n-1)/2$ as a module over
$\Z[c_1V,\ldots,c_nV]$.

For statement (2), observe that $CH^i(BG_k,M)$ is killed by the order
of $G$, by pullback and pushforward along $EG_k\to BG_k$.
Therefore, it suffices to prove the desired bounds for generators
of $CH^*(BG_k,M)/l$ for each prime number $l$. Use that
for each finite group $G$,
the graded ring $CH^*(BG)/l$
has Castelnuovo--Mumford regularity at most 0
\cite[Theorem 6.5]{Totarobook}. Under the assumptions of (2),
it follows that $CH^*(BG)/l$ is generated by elements
of degree at most $\sum(|x_i|-1)$ as a module
over $\F_l[x_1,\ldots,x_m]$ \cite[Lemma 3.10, Theorem 3.14]{Totarobook}.
The same bound applies to every subgroup of $G$. As in part (1),
it follows that $CH^*(BG,M)/l$ is generated by elements
of degree at most $\sum(|x_i|-1)$ as a module over $\F_l[x_1,\ldots,x_m]$.

Likewise, statement (3) is known for $M=\Z$ \cite[Lemma 6.3]{Totarobook}.
By Theorem \ref{transfer}, the same bound holds
for any $\Z G$-module $M$.
\end{proof}

Thus we have strong bounds for generators of the twisted Chow groups
of a finite group, stronger than what is true for cohomology.
We may ask the same question about relations.

\begin{question}
Let $G$ be a finite group, with a faithful complex representation $V$
of dimension $n$. By Theorem \ref{degreebound}, $CH^*(BG_{\C},M)$
is generated in degrees at most $n(n-1)/2$ as a module over
$\Z[c_1V,\ldots,c_nV]$, for all $\Z G$-modules $M$. Is there also a bound
for the degrees of relations in $CH^*(BG_{\C},M)$ that depends only on $G$?
\end{question}

One natural approach fails: the Castelnuovo--Mumford regularity
of $CH^*(BG,M)$ can be arbitrarily large, for a fixed group $G$ (Remark
\ref{regularity}).

\section{The group $\Z/2\times \Z/2$}

We now compute the Chow groups of $G=\Z/2\times\Z/2$ with coefficients
in any $\F_2 G$-module. Here $\Z/2\times \Z/2$
is the simplest finite group with $p$-rank
greater than 1, and hence with non-periodic cohomology. The calculations
show some new phenomena (Remarks \ref{kleinremark}
and \ref{regularity}, and Theorem \ref{counterexample}). To understand
the statement, note that for any field $k$ of characteristic not 2,
$CH^*BG_k$ is isomorphic to $\Z[u,v]/(2u,2v)$, where $u$ and $v$ are
first Chern classes of 1-dimensional representations of $G$
\cite[Theorem 2.10 and Lemma 2.12]{Totarobook}.

\begin{theorem}
\label{klein}
Let $G=\Z/2\times \Z/2$, and let $M$ be an
$\F_2 G$-module. Let $k$ be a separably closed field of characteristic not 2.
Then
$$CH^*(BG_k,M)\cong \im(M^G\otimes_{\Z} CH^*BG_k\to H^*(BG,M)).$$
\end{theorem}

\begin{proof}
First let $k=\C$. It suffices to show two statements: the product map
$$M^G\otimes CH^*BG = CH^0(BG,M)\otimes CH^*BG\to CH^*(BG,M)$$
is surjective, and $CH^*(BG,M)$ injects into $H^*(BG,M)=
H^*(G,M)$ (using the classical topology). Let $F$ be an algebraic closure
of $\F_2$. Replacing $M$ by $M\otimes_{\F_2}F$ changes all these groups
by tensoring up to $F$. Therefore, it suffices to prove both statements
for $M\otimes_{\F_2}F$. That is, we can assume from now on that $M$
is an $FG$-module.

Write $G=\langle g,h:g^2=1, h^2=1, gh=hg\rangle$.
Let $L_1$ and $L_2$ be the 1-dimensional representations of $G$
over $k$ given by $g\mapsto -1$, $h\mapsto 1$ (for $L_1$)
and $g\mapsto 1$, $h\mapsto -1$ (for $L_2$). Let $u=c_1L_1$ and $v=c_1L_2$
in $CH^1BG$; then $CH^*BG=\Z[u,v]/(2u,2v)$.
The representation $L_1\oplus L_2$ of $G$ is faithful, and its Chern
classes are polynomials in $u$ and $v$. By Theorem \ref{degreebound}(2),
it follows that $CH^*(BG,M)$ is generated by elements of degree 0
as a module over $CH^*BG$. That is, the product map
$$M^G\otimes CH^*BG\to CH^*(BG,M)$$
is surjective.

For later use, the cohomology ring of $G$ is the polynomial ring
$H^*(BG,F)\cong F[x,y]$ with $|x|=|y|=1$. We can choose the generators
so that the cycle map $CH^*(BG)\otimes F\to H^*(BG,F)$ takes $u$ to $x^2$
and $v$ to $y^2$.

It remains to show that $CH^*(BG,M)\to H^*(BG,M)$ is injective.
These groups commute with direct limits of $FG$-modules;
so we can assume that $M$ is an $FG$-module of finite dimension over $F$.
Using direct sums, we can also assume that $M$ is indecomposable.
Then we can use the classification of indecomposable
$FG$-modules, as follows \cite[v.~1, Theorem 4.3.3]{Benson}.
For a finite-dimensional
$FG$-module $M$, define the syzygy module $\Omega M$ as the kernel of
the surjection from a projective cover of $M$ to $M$;
then $\Omega M$ is well-defined
up to isomorphism. (Likewise, define the shift
$\Omega^{-1}M$ as the cokernel
of the inclusion from $M$ to its injective hull.)
Each element $f$ of $H^n(G,F)$ with $n>0$
is represented by a map $\Omega^n F\to F$ of $FG$-modules;
for $f\neq 0$,
define $L_f$ to be the kernel. Then, for $G=\Z/2\times\Z/2$,
every indecomposable $FG$-module
is isomorphic to either $FG$, $\Omega^n F$ for some $n\in\Z$,
or $L_{\zeta^n}$ for some $n>0$ and some nonzero $\zeta\in H^1(G,F)$.
Here $\zeta$ only matters up to scalars, and so the last type of module
is determined (up to isomorphism) by $n>0$ and
a point in $\P(H^1(G,F))\cong \P^1(F)$. Here
$\Omega^n F$ has dimension $2|n|+1$ and $L_{\zeta^n}$ has dimension
$2n$.

In each case, we use that $CH^0(BG,M)=H^0(BG,M)=M^G$
and $CH^1(BG,M)\to H^2(BG,M)$ is injective
(Corollary \ref{injectivecor}).
First, let $M=FG$. Then $H^2(G,M)=0$, and so $CH^1(BG,M)=0$.
Since $M^G\otimes_F CH^*BG\to CH^*(BG,M)$ is surjective
(and $CH^*(BG)\otimes F=F[u,v]$ is generated in degree 1),
it follows that $CH^i(BG,M)=0$ for $i>0$. Thus $CH^*(BG,M)\to
H^*(BG,M)$ is injective.

It will be useful to recall the description
of Tate cohomology for finite groups \cite[section VI.4]{Brown}:
$\widehat{H}^j(G,M)$ is isomorphic to $H^j(G,M)$
if $j>0$, to $H_{-1-j}(G,M)$ if $j\geq -2$, and $\widehat{H}^{-1}(G,M)$
and $\widehat{H}^0(G,M)$ are the kernel and cokernel of the trace map:
$$0\to\widehat{H}^{-1}(G,M)\to M_G\xrightarrow[\tr]{}
M^G\to \widehat{H}^0(G,M)\to 0.$$

For each indecomposable $FG$-module other than $FG$, I claim that the trace map
$\tr\colon M\to M$ is zero. (This holds more generally for any $p$-group.)
If not, let $x$ be an element of $M$
with $\tr(x)\neq 0$. Then there is an $FG$-linear map $f\colon FG\to M$
that takes 1 to $x$, and hence $\tr(1)$ to $\tr(x)\neq 0$. But $\tr(1)$ in $FG$
spans the socle, $(FG)^G\cong F$. It follows that $f$ is injective. Since
the $FG$-module $FG$ is injective as well as projective
\cite[Proposition 3.1.10]{Benson}, it follows that
$M$ contains $FG$ as a summand. Thus, for $M$ indecomposable and not
isomorphic to $FG$, the trace is zero on $M$. Equivalently, $M^G=H^0(G,M)$
maps isomorphically to $\widehat{H}^0(G,M)$. This is relevant because
we have more direct access to $\widehat{H}^0(G,M)$ in the following
calculations, and hence we can read off $M^G$.

Next, let $M=F$. We have $H^*(BG,F)
\cong F[x,y]$, and $CH^*(BG)\to H^*(BG,F)$ sends $u\mapsto x^2$
and $v\mapsto y^2$. It follows that
$CH^*(BG,M)=F[u,v]$ injects into $H^*(BG,F)$, as we want.

Next, let $M=\Omega^{-m}F$ with $m>0$. Then $\widehat{H}^i(G,M)
\cong \widehat{H}^{i+m}(G,F)\cong F^{i+m+1}$ for $i\geq 0$.
In particular,
$CH^1(BG,M)$ is the image of $M^G\otimes CH^1BG$ in $H^2(BG,M)$,
thus the image of $F\{ x^m,x^{m-1}y,\ldots,y^m\}\otimes_F F\{x^2,y^2\}$,
which is all of $H^2(BG,M)\cong F\{ x^{m+2},x^{m+1}y,\ldots,y^{m+2}\}$.
Therefore, $CH^*(BG,M)$ is a quotient of the $F[u,v]$-module
$F[u,v]\{ e_0,\ldots,e_m\}/(ue_{i+2}-ve_{i})$ for $0\leq i \leq m-2$,
where $e_i$ maps to $x^{m-i}y^i$ in $M^G/\tr(M)$. But we compute that this
$F[u,v]$-module maps isomorphically to $H^{\ev}(BG,M)$ (that is,
to the subspace of $F[x,y]$ spanned by homogeneous polynomials of degree
at least $m$ and congruent to $m$ modulo 2). Therefore,
$CH^*(BG,M)\to H^{\ev}(BG,M)$ is an isomorphism (hence injective,
as we want).

Next, let $M=\Omega^mF$ with $m>0$. Then $\widehat{H}^i(G,M)
\cong \widehat{H}^{i-m}(G,F)$. These vector spaces decrease
from dimension $m$ (when $i=0$) to dimension 1 (when $i$ is $m-1$
or $m$) and then increase again, by the description of Tate cohomology
above.
For $i\geq 0$ and $j<-i$, the product
$\widehat{H}^i(G,F)\times \widehat{H}^j(G,F)\to \widehat{H}^{i+j}(G,F)$
can be identified with
the cap product of cohomology with homology (which is dual to the product
on cohomology in positive degrees, hence usually nonzero).
On the other hand, products
from negative degree into nonnegative degree are zero. Namely,
Benson and Carlson showed that for $i>0$ and $-i\geq j<0$,
for $G=\Z/2\times \Z/2$ (as for many other groups of $p$-rank at least 2),
the product
$\widehat{H}^i(G,F)\times \widehat{H}^j(G,F)\to \widehat{H}^{i+j}(G,F)$
is zero \cite[Lemma 2.1 and Theorem 3.1]{BC}.

For $m=1$ (that is, $M=\Omega F$), it follows that the image of the product
map $M^G\otimes CH^1BG\to H^2(BG,M)$ is zero, and hence $CH^1(BG,M)=0$.
Since $M^G\otimes CH^*BG\to CH^*(BG,M)$ is surjective (and $CH^*BG$
is generated in degree 1), it follows that $CH^i(BG,M)=0$ for all $i>0$.
Thus $CH^*(BG,M)\to H^*(BG,M)$ is injective, as we want.

For $m>1$, let $R=CH^*(BG)\otimes F=F[u,v]$. Then in degrees at most 1,
$CH^*(BG,M)$ agrees with the $R$-module $N:=R\{e_0,\ldots,e_{m-1}\}
/(ue_{i+2}=ve_i \text{ for }0\leq i\leq m-3, ue_0=0, ue_1=0,
ve_{m-2}=0, ve_{m-1}=0)$, using that $CH^1(BG,M)$ is the image of
$M^G\otimes CH^1BG\to H^2(BG,M)$. Since $M^G\otimes CH^*BG\to CH^*(BG,M)$
is surjective in all degrees, we have a surjection of $R$-modules
from $N$ to $CH^*(BG,M)$. But we compute that $N^i$ maps isomorphically
to $H^{2i}(BG,M)$ for $0\leq i<m/2$, and $N$ is zero in higher degrees.
In particular, $N$ injects into $H^*(BG,M)$. Therefore, $N$ maps
isomorphically to $CH^*(BG,M)$, and $CH^*(BG,M)$ injects into
$H^*(BG,M)$ as we want.

Finally, let $M$ be the $FG$-module $L_{\zeta^n}$, for a positive integer
$n$ and a nonzero element $\zeta$ in $H^1(BG,F)=F\{ x,y\}$.
Choose one of $x$ or $y$ which is linearly independent of $\zeta$,
say $x$ (without loss of generality).
Then $\widehat{H}^*(G,M)$ is periodic with period 1; namely,
multiplication by $x$ is an isomorphism on this $F[x,y]$-module
\cite[v.~1, Corollary 5.10.7]{Benson}.

To describe the cohomology of $M=L_{\zeta^n}$
in more detail: by the exact sequence
$$0\to M\to \Omega^nF\to F\to 0,$$
we have $H^{i+1}(G,M)\cong H^i(G,F)/\zeta^n H^{i-n}(G,F)$
for $i\geq n-1$. (This uses that $\zeta^n$ is a non-zero-divisor
in $H^*(G,F)=F[x,y]$.) Since $\zeta$ is linearly independent of $x$,
we can also view $H^*(G,F)$ as the polynomial ring $F[x,\zeta]$.
Using the periodicity of $\widehat{H}^*(G,M)$, let us identify
$M^G=\widehat{H}^0(G,M)$ with $H^n(G,M)\cong H^{n-1}(G,F)$.
As such, $M^G$ has a basis $e_i = x^{n-1-i}\zeta^i$ with
$0\leq i\leq n-1$. Let $w=\zeta^2$; then $CH^*(BG)\otimes F=F[u,w]$.
Then $CH^1(BG,M)\cong \im(M^G\otimes_F (CH^1(BG)\otimes F)
\to H^2(BG,M))$ is spanned by $ue_i$ and $we_i$ for $0\leq i\leq n-1$,
modulo the relations $ue_{i+2}=we_{i}$ for $0\leq i\leq n-3$,
$we_{n-2}=0$, and $we_{n-1}=0$. 

Let $R=CH^*(BG)\otimes F=F[u,w]$.
Since $M^G\otimes_F R\to CH^*(BG,M)$ is surjective,
$CH^*(BG,M)$ is a quotient of the graded $R$-module
$$N:=R\{ e_0,\ldots,e_{n-1}\}/(ue_{i+2}=we_{i}
\text{ for }0\leq i\leq n-3, we_{n-2}=0, we_{n-1}=0).$$
But we compute that this module $N$ maps isomorphically
to $H^{\ev}(BG,M)$, viewed as the part of $F[x,\zeta]/(\zeta^n)$
in degrees at least $n-1$. Therefore, $N$ maps isomorphically
to $CH^*(BG,M)$, and $CH^*(BG,M)$ maps isomorphically
to $H^{\ev}(BG,M)$, hence injectively, as we want.

That completes the proof for $k=\C$. More generally, let $k$
be a separably closed field of characteristic not 2. Then
the Hochschild--Serre spectral sequence for $EG_k\to BG_k$ shows
that $H^*_{\et}(BG_k,M)$ is isomorphic to $H^*(G,M)$, for every
$\F_2 G$-module $M$ \cite[Theorem III.2.20]{Milne}.
Also, Theorem \ref{injective}
shows that $CH^1(BG_k,M)\to H^2_{\et}(BG_k,M(1))$ is injective.
(The twist here is irrelevant, since the \etale sheaf
$\mu_2$ is canonically isomorphic to $\Z/2$ over $k$.) Given that,
the arguments over $\C$ work without change over $k$.
\end{proof}

\begin{remark}
\label{kleinremark}
For the group $G=\Z/2\times \Z/2$, Theorem \ref{klein}
shows that $CH^*(BG_{\C},M)$
maps injectively to $H^*(BG,M)$, for all $\F_2 G$-modules $M$.
But for some modules $M$,
this map is far from an isomorphism. In particular, for $m>0$,
we have shown that
$CH^*(BG,\Omega^m \F_2)$ is zero in degrees at least $m/2$, whereas
$H^*(BG,\Omega^m \F_2)$ contains $H^*(BG,\F_2)=\F_2[x,y]$ (shifted in degree)
as a submodule. Thus the ``support variety'' of $H^*(BG,M)$ is all
of $\Spec H^*(BG,\F_2)=A^2_{\F_2}$, while the support variety of $CH^*(BG,M)$
is only the origin in $\Spec CH^*(BG)/2=A^2_{\F_2}$.
\end{remark}

\begin{remark}
\label{regularity}
For a finite group $G$ and a prime number $p$,
the Castelnuovo--Mumford regularity
of $CH^*(BG_{\C})/p$ is at most zero
\cite[Theorem 6.5]{Totarobook}. In terms of a faithful representation
$V$ of $G$ over $\C$, with $n:=\dim(V)$, this regularity bound
amounts to an upper bound for the degrees
of generators, relations, relations between relations, and so on,
for $CH^*(BG_{\C})/p$ as a graded module over the Chern classes
$\F_p[c_1V,\ldots,c_nV]$.

We have seen that there is a bound for the degrees of generators
of $CH^*(BG_{\C},M)$ as a module over $\F_p[c_1V,\ldots,c_nV]$
for all $\F_p G$-modules $M$, depending only on $G$
(Theorem \ref{degreebound}). However,
the regularity of $CH^*(BG_{\C},M)$ does not have such a bound.
Take $G=\Z/2\times \Z/2$ and $M=\Omega^m\F_2$ for $m\geq 2$.
Let $R=CH^*(BG)/2=\F_2[u,v]$. By the properties of regularity,
$CH^*(BG,M)$ has the same regularity over $R$ as over the Chern
classes of a faithful representation \cite[Lemma 3.10]{Totarobook}.

By the calculation of $CH^*(BG,M)$ in the proof of Theorem
\ref{klein} plus the Hilbert syzygy theorem \cite[Corollary 19.7]{Eisenbud},
$CH^*(BG,M)$ has a graded free
resolution over $R$ of the form
$$0\to R^{\oplus 2}\to R^{\oplus m+2}\to R^{\oplus m}\to CH^*(BG,M)\to 0.$$
Here the generators of $R^{\oplus m}$ are in degree 0, and the generators
of $R^{\oplus m+2}$ (corresponding to the relations in $CH^*(BG,M)$)
are in degree 1. From the Hilbert series of $CH^*(BG,M)$,
we compute that the module of ``relations between relations'' $R^{\oplus 2}$
has generators in degrees $\lfloor (m+2)/2\rfloor$
and $\lfloor (m+3)/2\rfloor$. Therefore, $CH^*(BG,M)$ has
Castelnuovo--Mumford regularity $\lfloor (m+3)/2\rfloor-2
=\lfloor (m-1)/2\rfloor$ \cite[Theorem 3.14]{Totarobook}.
In particular, the regularity
of $CH^*(BG,M)$ cannot be bounded in terms of $G$.
\end{remark}

\section{Twisted Chow groups vs.\ twisted motivic cohomology}

We now show that the surjection from 
twisted motivic cohomology $H^{2i}_{\M}(X,M(i))$
to twisted Chow groups $CH^i(X,M)$ is not always
an isomorphism.
As a result, one might think that the definition of twisted Chow
groups from section \ref{definition} is ``wrong'', and that the definition
should be changed to agree with twisted motivic cohomology.
Given the good properties of twisted Chow groups
from section \ref{definition}, however,
I believe that twisted Chow groups are worth studying. They form a nontrivial
intermediary between twisted motivic cohomology
and twisted \etale cohomology.
An advantage of twisted Chow groups
is that $CH^1(X,M)$ injects into \etale motivic cohomology
$H^2_{\et}(X,M(1))$, whereas (as we will see)
$H^2_{\M}(X,M(1))$ does not always inject into $H^2_{\et}(X,M(1))$.

\begin{theorem}
\label{counterexample}
(1) There is a smooth complex variety $X$
with a locally constant \etale sheaf $M$ such that the maps
$$H^2_{\M}(X,M(1))\to H^2_{\et}(X,M(1))$$
and
$$H^2_{\M}(X,M(1))\to CH^1(X,M)$$
are not injective. 

(2) There is a short exact sequence
$0\to A\to B\to C\to 0$ of locally constant \etale sheaves on $X$
such that $A$ is coflasque but the sequence
$$CH^1(X,A)\to CH^1(X,B)\to CH^1(X,C)$$
is not exact.
\end{theorem}

\begin{proof}
Let $k=\C$.
The idea is to compare twisted Chow groups with twisted motivic
cohomology for $BG_k$ with $G=\Z/2\times \Z/2
=\langle g,h:g^2=1,h^2=1,gh=hg\rangle$. (This is the smallest
group $G$ that has a coflasque $\Z G$-lattice that is not invertible.
That is relevant because of Theorem \ref{coflasque}.)
We can take the smooth
variety $X$ to be $U/G$ for any open subset $U$ of a representation $V$
of $G$ over $k$ such that $G$ acts freely on $U$
and $V-U$ has codimension at least 2 in $V$.

Let $M$ be the $G$-module $\Omega^{-m}\F_2$ with $m\geq 2$.
Then $M$ is the vector space $(\F_2)^{2m+1}$ with basis
$e_1,\ldots,e_{2m+1}$, and $G$ acts by:
$g(e_i)=e_i$ for $1\leq i\leq m+1$,
$g(e_{m+1+i})=e_i+e_{m+1+i}$ for $1\leq i\leq m$,
$h(e_i)=e_i$ for $1\leq i\leq m+1$, and
$h(e_{m+1+i})=e_{i+1}+e_{m+1+i}$ for $1\leq i\leq m$
\cite[v.~1, Theorem 4.3.3]{Benson}.

Then $M^G$ is spanned by $e_1,\ldots,e_{m+1}$. Let $H_1=\langle g\rangle$,
$H_2=\langle h\rangle$, and $H_3=\langle gh\rangle$
be the subgroups of order 2 in $G$. For $a=1,2,3$,
the subspace $M^{H_a}$ is also spanned by $e_1,\ldots,e_{m+1}$.
Define a $\Z G$-linear map from the permutation module
$B:=(\Z G)^m\oplus \Z^{m+1}$ to $M$,
sending the $\Z G$-module generators of $B$ by:
$f_i\mapsto e_{m+1+i}$ for $1\leq i\leq m$
and $f_{m+i}\mapsto e_i$ for $1\leq i\leq m+1$.
Then $B^H\to M^H$ is surjective for every subgroup $H$ of $G$,
and so the kernel $A$ is a coflasque $\Z G$-lattice.

A $\Z$-basis for $A\cong \Z^{5m+1}$ is given by the elements $s_i=2f_i$
for $1\leq i\leq 2m+1$, $s_{2m+1+i}=gf_i-f_i-f_{m+i}$
for $1\leq i\leq m$, $s_{3m+1+i}=hf_i-f_i-f_{m+1+i}$
for $1\leq i\leq m$, and $s_{4m+1+i}=ghf_i-f_i-f_{m+i}-f_{m+1+i}$
for $1\leq i\leq m$. Here $A^G$ has rank $2m+1$
and $A^{H_a}$ has rank $3m+1$ for $a=1,2,3$. In more detail,
a $\Z$-basis for $A^G$ is given by $s_{m+1},\ldots,s_{2m+1}$,
and $2s_i+s_{2m+1+i}+s_{3m+1+i}+s_{4m+1+i}$ for $1\leq i\leq m$.
A $\Z$-basis for $A^{H_1}$ is given by the basis for $A^G$ together
with $s_i+s_{2m+1+i}$ for $1\leq i\leq m$. For $A^{H_2}$,
we have the basis for $A^G$ together with $s_i+s_{3m+1+i}$
for $1\leq i\leq m$,
and for $A^{H_3}$ we have the basis for $A^G$
together with $s_i+s_{4m+1+i}$ for $1\leq i\leq m$,

Define a $\Z G$-linear map from $P:=\oplus_{a=1}^3 \Z[G/H_a]^{\oplus m}$
to $A$, sending the $\Z G$-module generators to $s_i+s_{2m+1+i}$,
$s_i+s_{3m+1+i}$, and $s_i+s_{4m+1+i}$. We read off that
$P^H\to A^H$ is surjective for every subgroup $H$ of $G$.
Therefore, the kernel of $P\to A$ is coflasque.
By Lemma \ref{motexact}, it follows that
$H^{2i}_{\M}(BG,P(i))\to H^{2i}_{\M}(BG,A(i))$ is surjective.
Using that lemma again for the coflasque resolution $0\to A\to B\to
\Omega^{-m}\F_2\to 0$, we find that
$$H^{2i}_{\M}(BG,P(i))\to H^{2i}_{\M}(BG,B(i))\to
H^{2i}_{\M}(BG,\Omega^{-m}\F_2(i))\to 0$$
is exact.

Since $P$ and $B$ are permutation modules, we can rewrite this exact
sequence as
$$\oplus_{a=1}^3 CH^i(BH_a)^{\oplus m}\to
CH^i(\Spec k)^{\oplus m}\oplus CH^i(BG)^{\oplus m+1}
\to H^{2i}_{\M}(BG,\Omega^{-m}\F_2(i))\to 0.$$
For $i>0$, we have $CH^i(\Spec k)=0$, and the maps
from $CH^iBH_a$ to $CH^iBG$ are multiples of the transfer map.
But transfer from $CH^iBH_a$ to $CH^iBG$ is zero for $i>0$, using
that these groups are killed by 2 and restriction from $CH^iBG$
to $CH^iBH_a$ is surjective.
Therefore, we have an isomorphism
$$CH^i(BG)^{\oplus m+1}\cong H^{2i}_{\M}(BG,\Omega^{-m}\F_2(i))$$
for $i>0$. (By inspection, this also holds for $i=0$.)

In particular, $H^2_{\M}(BG,\Omega^{-m}\F_2(1))\cong
(\F_2)^{2m+2}$. By Theorem \ref{klein}, we have $CH^1(BG,
\Omega^{-m}\F_2)\cong (\F_2)^{m+3}$. Thus, for $m\geq 2$,
the surjection 
$$H^2_{\M}(BG,\Omega^{-m}\F_2(1))\to CH^1(BG,
\Omega^{-m}\F_2)$$
from Corollary \ref{surjective}
is not injective. The map
$$H^2_{\M}(BG,\Omega^{-m}\F_2(1))
\to H^2_{\et}(BG,\Omega^{-m}\F_2(1))$$
factors through $CH^1(BG,
\Omega^{-m}\F_2)$, and so it is also not injective. We have now proved
two parts of the theorem.

It remains to show that the sequence $CH^1(X,A)\to CH^1(X,B)
\to CH^1(X,\Omega^{-m}\F_2)$ is not exact,
even though $A$ is coflasque. Because the surjection $P\to A$
has coflasque kernel, we know that $CH^1(X,P)\to CH^1(X,A)$
is surjective (Theorem \ref{coflasque}). So it is equivalent to show
that $CH^1(X,P)\to CH^1(X,B)\to CH^1(X,\Omega^{-m}\F_2)$ is not exact.
Since $P$ and $B$ are permutation modules,
we have to show that the sequence
$$\oplus_{a=1}^3 CH^1(BH_a)^{\oplus m}\to
CH^1(\Spec k)^{\oplus m}\oplus CH^1(BG)^{\oplus m+1}
\to CH^{1}(BG,\Omega^{-m}\F_2)\to 0.$$

The first map is a linear combination of transfers from the subgroups
$H_a$ to $G$, and so it is zero, as shown above. Therefore, we want
to show that $CH^1(BG)^{\oplus m+1}
\to CH^{1}(BG,\Omega^{-m}\F_2)$ is not an isomorphism.
The first group is isomorphic to $(\F_2)^{2m+2}$ and the second is
$(\F_2)^{m+3}$. Since $m\geq 2$, this is not an isomorphism.
The third part of the theorem is proved.
\end{proof}

\begin{remark}
\label{comparison}
Let us compare the advantages of twisted Chow groups $CH^i(X,E)$
and twisted motivic cohomology $H^{2i}_{\M}(X,E(i))$. Assume
here that $E$ is a locally constant \etale sheaf on a smooth
variety $X$ over a field $k$.

(1) When $E$ is the constant sheaf corresponding to an abelian group,
both theories agree with the usual Chow groups, $CH^i(X)\otimes_{\Z}E$.

(2) Twisted motivic cohomology has a long exact sequence
associated to a short exact sequence of locally
constant \etale sheaves $0\to A\to B\to C\to 0$ if $A$ is coflasque
(Lemma \ref{motexact}).
For twisted Chow groups, we can only say that $CH^i(X,A)\to CH^i(X,B)
\to CH^i(X,C)$ is exact if $A$ is invertible (Theorems
\ref{coflasque} and \ref{counterexample}).

(3) The cycle map $CH^1(X,E)\to H^2_{\et}(X,E(1))$ is injective, whereas
$H^2_{\M}(X,E(1))\to H^2_{\et}(X,E(1))$ need not be injective
(Theorems \ref{injective} and \ref{counterexample}).
More broadly, twisted Chow groups should be closer to \etale
cohomology than twisted motivic cohomology is.
\end{remark}

\appendix

\section{The residue on \'etale motivic cohomology}
\label{residuesection}

Here we construct the residue homomorphism on \etale
motivic cohomology twisted by a locally constant sheaf $E$,
$$\partial_v\colon H^a(F,E(a))\to H^{a-1}(k(v),E(a-1))$$
(Corollary \ref{residue}). 
This requires extra effort when $k(v)$ has characteristic $p>0$
and $p$ does not act invertibly on $E$. We use the residue homomorphism
in section \ref{definition} to show
that Chow groups with twisted coefficients
have the desired formal properties in full generality.
This appendix uses no results from the rest of the paper.

\begin{lemma}
\label{zero}
Let $O_v$ be a discrete valuation ring, and let
$i\colon \Spec k(v)\to \Spec O_v$ be the inclusion
of the closed point. For each $a\geq 1$,
$$\Z(a-1)_{k(v)}[-2]
\cong \tau_{\leq a+2} i^{!}\Z(a)$$
in $D_{\et}(k(v))$.
\end{lemma}

Here $i^{!}$ is the exceptional
inverse image functor
on derived categories, sometimes called $Ri^{!}$.

\begin{proof}
First, $\Z(a-1)_{k(v)}$ is concentrated in degrees at most $a-1$
in $D_{\et}(k(v))$, and so $\Z(a-1)[-2]$ is concentrated
in degrees at most $a+1$. Next,
using the Bloch-Kato conjecture (Voevodsky's theorem),
Geisser showed that the canonical map
$$\Z(a-1)_{k(v)}[-2]
\to \tau_{\leq a+1} i^{!}\Z(a)$$
in $D_{\et}(k(v))$ is an isomorphism.
Also, the truncation is unnecessary after inverting the exponential
characteristic $e$ \cite[Theorem 1.4]{Geisser}. (These results imply
the lemma when $k(v)$ has characteristic zero.)

It remains to show
that $M:=\h^{a+2}(i^{!}\Z(a)_{O_v})$ is zero.
Here $M$ is an \etale sheaf on $\Spec k(v)$, and so it suffices
to show that the stalk of $M$ at the separable closure $k(v)_s$ is zero.
This stalk is isomorphic to $H^{a+2}(k(v)_s,i^{!}\Z(a)_{O_{\nr}})$,
where $O_{\nr}$ is the maximal unramified extension of $O_v$
(that is, the strict henselization of $O_v$), and we use the same name $i$
for the inclusion $i\colon \Spec k(v)_s\to \Spec O_{\nr}$
\cite[Theorem 03Q9]{Stacks}. Let $j\colon\Spec F_{\nr}
\to \Spec O_{\nr}$ be the inclusion of the generic point.
There is an exact triangle
$$i_*i^{!}E\to E\to j_*j^*E$$
for every object $E$ in $D_{\et}(O_{\nr})$. Applying this
to $E=\Z(a)_{O_{\nr}}$, we have a long exact sequence
$$\cdots \to H^{j-1}_{\et}(F_{\nr},\Z(a))\to H^j_{\et}(k(v)_s,i^{!}\Z(a))
\to H^j_{\et}(O_{\nr},\Z(a))\to H^j_{\et}(F_{\nr},\Z(a))\to\cdots.$$

Consider the map from Zariski to \etale cohomology:
\small
$$\xymatrix@R-10pt@C-20pt{
H^{j-1}_{\Zar}(O_{\nr},\Z(a))\ar[r]\ar[d] &
H^{j-1}_{\Zar}(F_{\nr},\Z(a))\ar[r]\ar[d] &
H^{j}_{\Zar}(k(v)_s,i^!\Z(a))\ar[r]\ar[d] &
H^{j}_{\Zar}(O_{\nr},\Z(a))\ar[r]\ar[d] &
H^{j}_{\Zar}(F_{\nr},\Z(a))\ar[d] \\
H^{j-1}_{\et}(O_{\nr},\Z(a))\ar[r] &
H^{j-1}_{\et}(F_{\nr},\Z(a))\ar[r] &
H^{j}_{\et}(k(v)_s,i^!\Z(a))\ar[r] &
H^{j}_{\et}(O_{\nr},\Z(a))\ar[r] &
H^{j}_{\et}(F_{\nr},\Z(a))
}$$
\normalsize
Voevodsky proved the Beilinson-Lichtenbaum conjecture for smooth
schemes over a field, and Geisser deduced it for smooth schemes
over a discrete valuation ring (in particular, for a DVR itself)
\cite[Theorem 6.18]{VoevodskyBK},
\cite[Theorem 1.2(2)]{Geisser}. Thus, for both $O_{\nr}$ and
$F_{\nr}$, the map from Zariski to \etale cohomology is an isomorphism
for $j\leq a+1$ and injective for $j=a+2$. By the commutative diagram
above, it follows that $H^j_{\Zar}(k(v)_s,i^!\Z(a))\to H^j_{\et}
(Y,i^!\Z(a))$ is an isomorphism for $j\leq a+1$. But, by localization
in Zariski motivic cohomology, the first group is
$\cong H^j_{\Zar}(k(v)_s,\Z(a-1)[-2])\cong H^j_{\et}(k(v)_s,\Z(a-1)[-2])$
for $j\leq a+1$, using the Beilinson-Lichtenbaum conjecture
again. Thus, in $D_{\et}(k(v)_s)$, the map $\h^j(\Z(a-1)[-2])\to
\h^j(i^!\Z(a))$ is an isomorphism for $j\leq a+1$.

Next, consider the diagram above for $j=a+1$. In this case,
we have the extra information that $H^{a+2}_{\Zar}(O_{\nr},
\Z(a))=H^{a+2}_{\et}(O_{\nr},\Z(a))=0$,
because $O_{\nr}$ is strictly henselian and $\Z(a)$ is concentrated
in degrees at most $a$. Then the diagram above implies
that $H^{a+2}_{\Zar}(k(v)_s,i^!\Z(a))\to H^{a+2}_{\et}
(k(v)_s,i^!\Z(a))$ is an isomorphism. The first group is
$\cong H^{a+2}_{\Zar}(k(v)_s,\Z(a-1)[-2])=0$,
using that $\Z(a-1)[-2]$
is concentrated in degrees at most $a+1$. This completes
the proof that the map
$$\Z(a-1)[-2]\to \tau_{\leq a+2}i^!\Z(a)$$
in $D_{\et}(k(v))$ is an isomorphism.
\end{proof}

\begin{lemma}
\label{truncation}
Let $O_v$ be a discrete valuation ring,
and let $E$ be a locally constant \etale sheaf on $\Spec O_v$.
For each $a\geq 1$,
$$E(a-1)_{k(v)}[-2]
\cong \tau_{\leq a+1} i^{!}E(a)$$
in $D_{\et}(k(v))$.
\end{lemma}

\begin{proof}
Let $M=i^!\Z(a)$ in $D_{\et}(Y)$. 
Note that $E\otimes_{\Z}^L M$ is isomorphic to $i^!E(a)$,
and that $E(a-1)_{k(v)}[-2]$ is concentrated in degrees
at most $a+1$.
The result
follows from Lemma \ref{zero}, by the universal coefficient
theorem applied to stalks at any geometric point:
$$0\to E\otimes_{\Z} \h^{j}(M)
\to \h^{j}(E\otimes_{\Z}^L M))\to \Tor_1^{\Z}(E,
\h^{j+1}(M))\to 0.$$
\end{proof}

\begin{corollary}
\label{residue}
Let $O_v$ be a discrete valuation ring, $k(v)$ the residue field,
and $F$ the fraction field.
Let $E$ be a locally constant \etale sheaf on $\Spec O_v$.
Then, for each $a\geq 1$, we define a residue homomorphism
$$\partial_v\colon H^a(F,E(a))\to H^{a-1}(k(v),E(a-1)).$$
\end{corollary}

\begin{proof}
By the basic exact triangle for $i^{!}$,
there is a natural map $H^a(F,E(a))\to H^{a+1}(k(v),i^{!}E(a))$.
The latter group is (trivially) isomorphic to $H^{a+1}(k(v),\tau_{\leq a+1}
i^{!}E(a))$. By Lemma \ref{truncation}, that is isomorphic
to $H^{a+1}(k(v),E(a-1)[2])\cong H^{a-1}(k(v),E(a-1))$.
\end{proof}

\section{Purity for \etale motivic cohomology}
\label{puritysection}

We prove here some purity properties of \etale motivic
cohomology. The subtleties occur only for varieties
in characteristic $p>0$. The point is that we only have the localization
sequence in its usual form for \etale motivic cohomology
after inverting $p$
(by Cisinski--D\'eglise). Nonetheless, we prove some purity
results without inverting $p$, building on work of Geisser, Gros,
and Levine
\cite{Geisser, GL, Gros}. We use these results to define the \etale cycle map
for twisted Chow groups on regular schemes, without
inverting $p$ (Theorem \ref{etalecycle}).

\begin{lemma}
\label{lowcodim}
Let $X$ be a regular noetherian scheme of finite type
over a field $k$. Let $i\colon Y\to X$
be the inclusion of a regular subscheme of codimension $r$.
For each $a\geq r$, the canonical morphism
$$\Z(a-r)[-2r]\to \tau_{\leq a+r+1}i^!\Z(a)$$
is an isomorphism in $D_{\et}(Y)$.
\end{lemma}

\begin{proof}
We can reduce to the case where the field $k$ is perfect. Indeed,
if $k$ has characteristic zero, then $k$ is already perfect. If $k$ has
characteristic $p>0$, then
$X$ and $Y$ can be defined over some finitely generated field $k$ over
$\F_p$. We can view $k$ as the function field of a variety $B$
over $\F_p$. After shrinking $B$, $X$ is the generic fiber
of a regular scheme $U$ of finite type over $B$, and likewise
$Y$ is the generic fiber of a regular subscheme $V$ of codimension $r$
in $U$. Then $U$ and $V$ are smooth over the perfect field $\F_p$,
and it suffices to prove the lemma for $V$ inside $U$.

So we can assume that $X$ and $Y$ are smooth over a perfect field $k$.
In the Zariski topology, we have $\Z(a-r)[-2r]\cong i^!\Z(a)$
in $D_{\Zar}(Y)$;
that is a reformulation of the localization sequence for motivic
cohomology. This determines a morphism $\varphi\colon \Z(a-r)[-2r]\to
i^!\Z(a)$ in $D_{\et}(Y)$. The object $\Z(a-r)$ in $D_{\et}(Y)$ is concentrated
in degrees at most $a-r$, and so $\Z(a-r)[-2r]$ is concentrated
in degrees at most $a+r$.

By Cisinski and D\'eglise, $\varphi$ becomes
an isomorphism after inverting the exponential characteristic of $k$.
That completes the proof for $k$ of characteristic zero. So we now
assume that $k$ has characteristic $p>0$. Let $C$ be the cofiber
of $\varphi$. Then we know that
$C[1/p]=0$; that is, $\h^j(C)$ is $p$-power
torsion for each integer $j$.

Tensoring $\varphi$ over $\Z$ with $\F_p$
gives a morphism $\F_p(a-r)[-2r]\to i^!\F_p(a)$ in $D_{\et}(Y)$.
The object $\F_p(a)$ in $D_{\et}(X)$ is concentrated
in degree $p$; namely, by Geisser--Levine, it is isomorphic in $D_{\et}(X)$
to $\Omega^a_{X,\log}[-a]$, where $\Omega^a_{X,\log}$ is the subsheaf
of $\Omega^a_X$ generated locally by logarithmic differentials
$df_1/f_1\wedge\cdots df_a/f_a$ for units $f_1,\ldots,f_a$ \cite{GL}.
(This is a sheaf of $\F_p$-vector spaces, not an $O_X$-module.)

In thsee terms, using that $k$ is perfect,
Gros showed that the morphism $\F_p(a-r)[-2r]
\to \tau_{\leq a+r}i^!\F_p(a)$ in $D_{\et}(Y)$ is an isomorphism
\cite[eq.~II.3.5.3, Th.~II.3.5.8]{Gros}. (In his notation,
this is the statement
that $\underline{H}^j_Y(X,\Omega^a_{X,\log})$ is zero for $j<r$
and isomorphic to $\Omega^{a-r}_{Y,\log}$ for $j=r$.) By the octahedral
axiom for triangulated categories, we have an exact triangle
$\F_p(a-r)[-2r]\to i^!\F_p(a)\to C/p$ in $D_{\et}(Y)$. It follows that
$\h^j(C/p)=0$ for $j\leq a+r$, where in the case $j=a+r$
we use that $\h^{a+r+1}(\F_p(a-r)[-2r])=0$. By the exact triangle
$C\xrightarrow[p]{} C\to C/p$, it follows that multiplication by $p$
is an isomorphism on $\h^j(C)$ for $j\leq a+r$ and is injective
on $\h^{a+r+1}(C)$. But $\h^j(C)$ is $p$-power torsion for all $j$.
So $\h^j(C)=0$ for $j\leq a+r+1$. By definition of $C$,
it follows that $\h^j(\Z(a-r)[-2r])\to \h^j(i^!\Z(a))$
is an isomorphism for $j\leq a+r+1$. We conclude that
$$\Z(a-r)[-2r]\to \tau_{\leq a+r+1}i^!\Z(a)$$
is an isomorphism in $D_{\et}(Y)$.
\end{proof}

\begin{lemma}
\label{highcodim}
Let $X$ be a regular noetherian scheme of finite type
over a field $k$. Let $i\colon Y\to X$
be the inclusion of a closed subset of codimension 
at least $r$ everywhere.
For each $a<r$,
$$\tau_{\leq 2a+2}i^!\Z(a)=0$$
in $D_{\et}(Y)$.
\end{lemma}

\begin{proof}
As in the proof of Lemma \ref{lowcodim}, we can reduce
to the case where $X$ is smooth over a perfect
field $k$. In this case, we can stratify $Y$ into pieces that are smooth
over $k$. By induction, it suffices to consider the case
where $Y$ is a smooth subvariety of codimension at least $r$ in $X$.

By the localization sequence for \etale motivic cohomology,
with the exponential characteristic $e$ inverted,
for $a<r$, we have $Ri^!\Z(a)[1/e]\cong
\oplus_{l\leq e}\Q_l/\Z_l(a-r)[-1-2r]$. It follows
that $\tau_{\leq 2a+2}i^!\Z(a)$ becomes zero after inverting $e$.
That completes the proof for $k$ of characteristic zero.
So we can assume that $k$ has characteristic $p>0$.
In this case,
we have shown that $\h^j(i^!\Z(a))$ is $p$-power torsion for each
$j\leq 2a+2$.

We use again Geisser--Levine's result that
$\F_p(a)$ is isomorphic to $\Omega^a_{\log}[-a]$ in $D_{\et}(X)$.
Using that $a<r$, Gros showed (using that $a<r$)
that $\tau_{\leq a+r}i^!\F_p(a)=0$ \cite[eq.~II.3.5.3, Th.~II.3.5.8]{Gros}.
(In his terms, when $a<r$, he showed that $\underline{H}^j_Y(X,
\Omega^a_{\log})=0$
for $j\leq r$.) By the exact triangle $\Z(a)\xrightarrow[p]{} \Z(a)
\to \F_p(a)$ in $D_{\et}(X)$,  it follows that the $p$-torsion
subgroup of $\h^j(i^!\Z(a))$ is zero for $j\leq a+r+1$. We have shown
that this group is $p$-power torsion, and so in fact
$\h^j(i^!\Z(a))$ is zero for $j\leq a+r+1$. That is,
$\tau_{\leq a+r+1}i^!\Z(a)=0$ in $D_{\et}(Y)$. Since $a<r$,
it follows that $\tau_{\leq 2a+2}i^!\Z(a)=0$.
\end{proof}

\begin{corollary}
\label{Ecodim}
Let $X$ be a regular scheme of finite type
over a field $k$. Let $E$ be a locally constant
\etale sheaf on $X$. 

(1) Let $i\colon Y\to X$
be the inclusion of a regular subscheme of codimension $r$.
For each $a\geq r$, the canonical morphism
$$E(a-r)[-2r]\to \tau_{\leq a+r}i^!E(a)$$
is an isomorphism in $D_{\et}(Y)$.

(2) Let $i\colon Y\to X$
be the inclusion of a closed subset of codimension 
at least $r$ everywhere.
For each $a<r$,
$$\tau_{\leq 2a+1}i^!E(a)=0$$
in $D_{\et}(Y)$.
\end{corollary}

\begin{proof}
By definition, $E(a)=E\otimes_{\Z}^L \Z(a)$ in $D_{\et}(X)$.
Since $\Z(a-r)[-2r]$ in $D_{\et}(Y)$ is concentrated in degrees
at most $(a-r)+2r=a+r$, so is $E(a-r)[-2r]$. Given that, (1) and (2)
follow from Lemmas \ref{lowcodim} and \ref{highcodim}, together
with the universal coefficient theorem (applied to the stalks
at any geometric point):
$$0\to E\otimes_{\Z} \h^j(i^!\Z(a))\to \h^j(i^!E(a))
\to \Tor_1^{\Z}(E,\h^{j+1}(i^!\Z(a)))\to 0.$$
Here we have used that $E\otimes_{\Z}^L i^!\Z(a)\cong i^!E(a)$.
\end{proof}

\section{Twisted motivic cohomology and twisted Chow groups:
conjectures}

In this section, we propose a general notion of ``twist''
which should make it possible to define twisted motivic cohomology
and twisted Chow groups; the two theories do not always agree.
Namely, it should be possible to twist by
any birational sheaf with transfers $E$
in the sense of Kahn--Sujatha \cite[Definition 2.3.1]{KS}.
Some cases have been worked out,
including the notion of twisting by an Azumaya algebra
\cite{KL, ENY}, as well as by a locally constant \etale sheaf.
This paper has focused on the latter case.

This section is not logically necessary for the rest of the paper.
In particular, in this section we assume that the exponential
characteristic $e$ of the base field $k$ acts invertibly on $E$,
in order to use the good properties of categories of motives
with $e$ inverted. (By definition, $e=1$
if $k$ has characteristic zero, and $e=p$ if $k$ has characteristic
$p>0$.) I hope that inverting $e$ can be avoided.
When $E$ is a locally constant
\'etale sheaf (the main focus of this paper),
section \ref{definition} defines twisted Chow groups without inverting $e$.

Let $X$ be a noetherian scheme of finite dimension.
Building on Voevodsky's ideas,
Cisinski and D\'eglise defined the derived category of motives
$DM(X)$ \cite[Definition 11.1.1]{CDbook}. The definition is based
on an abelian category $\Sh^{\tr}(X,\Z)$, the category
of Nisnevich sheaves with transfers
\cite[Definition 10.4.2]{CDbook}. These are Nisnevich
sheaves of abelian groups
on the category of smooth separated schemes of finite type
over $X$, with transfers for finite correspondences in a precise sense.

For a presheaf with transfers $E$ over $X$, we define
the {\it contraction } $E_{-1}$ (following Voevodsky) by
$$E_{-1}(Y):=\coker(E(Y\times A^1)\to E(Y\times G_m))$$
for $Y$ smooth over $X$. This is also a presheaf with transfers
\cite[Lecture 23]{MVW}, \cite[Definition 2.4.1]{KS}. (These references
assume that $X=\Spec k$ for a field $k$, but the same argument
applies.) If $E$ is a homotopy invariant sheaf with transfers,
then so is $E_{-1}$.

Define a homotopy invariant sheaf with transfers $E$
over $X$ to be {\it birational }if
$E_{-1}=0$. The name is motivated by Kahn--Sujatha's result that
for a perfect field $k$, a homotopy invariant
sheaf with transfers $E$ over $k$ has $E_{-1}=0$ if and only if
$E(Y)\xrightarrow[]{\cong} E(U)$ for every dense open subset $U$
of a smooth $k$-scheme $Y$
\cite[Proposition 2.5.2]{KS}. It is clear over any base scheme $X$
that a homotopy invariant sheaf with transfers that has the latter property
has $E_{-1}=0$, hence is ``birational'' in our sense.

Write $HI(X)$ for the full subcategory of homotopy invariant
Nisnevich sheaves with transfers.
By construction, there is a fully faithful functor
$HI(X)\to DM(X)$. As a result, we get a definition of motivic
cohomology twisted by an object $E$ in $HI(X)$:
$$H^i_{\M}(X,E(j)):=\Hom_{DM(X)}(1_X,E(j)[i]),$$
for integers $i$ and $j$. In particular, we have this definition
for $E$ a birational sheaf with transfers.

On the other hand, we can also define twisted Chow groups. One could define
these to be equal to $H^{2i}_{\M}(X,E(i))$; but we consider
a different notion, inspired by Rost's ideas, which mixes the \etale
and Zariski topologies. Then it becomes an interesting question
to compare twisted motivic cohomology and twisted Chow groups;
see the examples below.

The idea is that \etale sheafification gives a functor
from $HI(X)$ to $HI_{\et}(X)$, the category of homotopy invariant
\etale sheaves with transfers over $X$. (For $X=\Spec k$,
this is \cite[Theorem 6.17]{MVW}.) Moreover, $(E_{\et})_{-1}
\cong (E_{-1})_{\et}$, and so this functor takes birational Nisnevich
sheaves with transfer to birational \etale sheaves with transfer.
There is a tensor product on the abelian category
of \etale sheaves with transfer, and hence a derived tensor product
on the derived category of \etale sheaves with transfer,
written $\otimes^{\tr}_{L,\et}$, or $\otimes^{\tr}$ for short.
(These tensor products are part of the structure of ``premotivic
category'' constructed
in \cite[Corollary 2.1.12 and section 2.2.4]{CDetale}.)

Suppose that $X$ is a separated scheme of finite type over a field $k$.
In this case, Rost defined the abelian category of cycle modules
over $X$ \cite{Rost}. 

\begin{conjecture}
\label{cyclemodule}
Let $E$ be a birational Nisnevich sheaf with transfers over $X$.
For every field $F$ over $X$, define
$$H^*[E](F)=\oplus_{j\geq 0}H^j_{\et}(F,E(j)).$$
Then this is a cycle module over $X$.
\end{conjecture}

Here $\Z(j)$ denotes Voevodsky's motivic
cohomology complex, and $E(j):=E\otimes^{\tr}\Z(j)$. (The derived
tensor product is meant there, since $\Z(j)$ with $j\geq 0$ is a complex
of sheaves with transfer, not just a sheaf.
Indeed, $\Z(j)$ itself can be defined
as the derived tensor product of $j$ copies of $\Z(1)\cong G_m[-1]$,
$\Z(j)=\Z(1)\otimes^{\tr}\cdots\otimes^{\tr}\Z(1)$. Conjecture
\ref{cyclemodule} involves the \etale sheafification of these objects,
considered over fields. Part of the difficulty for the conjecture
is that the relation
between Rost's cycle modules and the derived category of motives
has only been worked out over a field, not over a more general
scheme $X$ \cite{Deglisegenerique, Deglisegeneric}.

Given the conjecture,
Rost's theory gives a definition of twisted Chow groups
$CH_i(X,E)$, meaning $A_i(X,H^*[E])_{-i}$ in Rost's notation. That is,
$CH_i(X,E)$ is defined as the cokernel of the residue homomorphism
$$\oplus_{x\in X_{(i+1)}}H^1_{\et}(k(x),E(1))
\to \oplus_{x\in X_{(i)}}H^0_{\et}(k(x),E).$$
When $X$ is smooth over $k$, we define $CH^i(X,E)$ likewise
in terms of codimension; so $CH^i(X,E)\cong CH_{n-i}(X,E)$
if $X$ is smooth of dimension $n$ everywhere.

One could define a different notion of twisted Chow groups
using the Nisnevich
rather than the \etale topology on fields; but that would be less interesting,
as it would always coincide with twisted motivic cohomology (in bidegree
$(2i,i)$). Our definition of twisted
Chow groups sits between twisted motivic cohomology
and twisted \etale cohomology, as follows. Given Conjecture
\ref{cyclemodule}, the proof is the same
as that of Corollary \ref{surjective} and Theorem \ref{etalecycle}.

\begin{lemma}
Let $X$ be a smooth scheme of finite type over a field $k$.
Assume Conjecture \ref{cyclemodule}. Then
for every birational Nisnevich sheaf with transfers $E$ over $X$,
we have natural homomorphisms
$$H^{2i}_{\M}(X,E(i))\to CH^i(X,E)\to H^{2i}_{\et}(X,E(i)),$$
the first of which is surjective.
\end{lemma}

Neither map is an isomorphism, in general.
In the following examples, let $X$ be a smooth scheme over $k$.

\begin{itemize}
\item
Let $E$ be the constant sheaf $\Z_X$. Then both $H^{2i}_{\M}(X,\Z(i))$
and $CH^i(X,\Z_X)$ can be identified with the usual Chow group, $CH^i(X)$.
(For $CH^i(X,\Z_X)$, this follows from the definition by generators
and relations: $H^0_{\et}(k(x),\Z_X)\cong \Z$ and
$H^1_{\et}(k(x),\Z_X(1))\cong H^0_{\et}(k(x),G_m)\cong k(x)^*$.)
The homomorphism from $CH^i(X)$ to $H^{2i}_{\et}(X,\Z(i))$
is rationally an isomorphism, but not integrally, in general.
For example, there are smooth complex varieties $X$
with $CH^2(X)/l$ infinite for a prime number $l$ \cite{Schoen,
TotaroChow},
whereas $H^{4}_{\et}(X,\Z(2))/l$ is
contained in $H^4_{\et}(X,\Z/l(2))$,
which is finite.
\item
Let $A$ be an Azumaya algebra over $X$, and let $E=\Z^A$
be the Nisnevich sheaf over $X$ associated to $K_0^A$.
By Kahn--Levine and Elmanto--Nardin--Yakerson, this can also be described
as the subsheaf
of $\Z_X$ that is the image of the rank homomorphism $K_0^A\to \Z_X$
\cite[Lemma 2.17]{ENY}. Then the \etale sheafification
of $\Z^A$ is simply $\Z_X$, because the Azumaya algebra $A$ is \'etale-locally
trivial. So the homomorphism $H^{2i}_{\M}(X,\Z^A(i))\to CH^i(X,\Z^A)$
is a homomorphism to the usual Chow groups, $CH^iX$. This homomorphism
was considered by Kahn and Levine \cite[section 5.9]{KL}.
It is not always
an isomorphism, even for $i=0$. Namely, for a smooth variety $X$
over $k$, the image of $H^0_{\M}(X,\Z^A(0))\to CH^0(X)=\Z$
is a subgroup of index equal to the index of $A$
over the function field $k(X)$, which can be greater than 1.
\item
Let $E$ be a birational {\it \'etale }sheaf with transfers over $X$.
Then $E$ is in particular a birational Nisnevich sheaf
with transfers over $X$, and so the definitions above apply.
In this case, $H^{2i}_{\M}(X,E(i))\to CH^i(X,E)$ is surjective,
by inspection of the generators of $CH^i(X,E)$. One main result
of this paper is that this surjection need not be
an isomorphism, even for $i=1$.

In particular, let $E$ be a locally constant \etale sheaf
over $X$. (We only consider sheaves of abelian groups.)
Then $E$ has transfers in a natural way;
for $X=\Spec k$, this is \cite[Lemma 6.11]{MVW}.
So every locally constant \etale sheaf $E$
can be viewed as a birational \etale sheaf
with transfers. We have seen that $H^2_{\M}(X,E(1))\to CH^1(X,E)$
need not be an isomorphism, even in this special case
(Theorem \ref{counterexample}).
\end{itemize}


\small \sc UCLA Mathematics Department, Box 951555,
Los Angeles, CA 90095-1555

totaro@math.ucla.edu
\end{document}